\documentclass[11pt,a4paper]{article}

\usepackage[margin=1in]{geometry}
\usepackage[T1]{fontenc}
\usepackage[utf8]{inputenc}
\usepackage{lmodern}
\usepackage{microtype}
\usepackage{amsmath,amssymb,amsthm}
\usepackage{mathtools,bm}
\numberwithin{equation}{section}

\usepackage{graphicx}
\usepackage{epstopdf}
\DeclareGraphicsExtensions{.pdf,.png,.jpg,.jpeg,.eps}
\usepackage{xcolor}
\usepackage{array,colortbl,booktabs,multirow}
\usepackage{arydshln}

\usepackage[ruled,lined,longend,algo2e]{algorithm2e}
\usepackage{hyperref}
\hypersetup{
  colorlinks=true,
  allcolors=green!50!black,
  urlcolor=red!50!black,
  pdftitle={A nonmonotone globalization framework for Anderson acceleration for contractive and nonexpansive fixed point problems},
  pdfauthor={Zekai Li and Wei Bian}
}
\usepackage[nameinlink,noabbrev]{cleveref}

\theoremstyle{plain}
\newtheorem{theorem}{Theorem}[section]

\newtheorem{proposition}[theorem]{Proposition}

\theoremstyle{definition}

\newtheorem{assumption}[theorem]{Assumption}

\theoremstyle{remark}
\newtheorem{remark}[theorem]{Remark}

\crefname{theorem}{Theorem}{Theorems}
\crefname{lemma}{Lemma}{Lemmas}
\crefname{proposition}{Proposition}{Propositions}
\crefname{corollary}{Corollary}{Corollaries}
\crefname{claim}{Claim}{Claims}
\crefname{definition}{Definition}{Definitions}
\crefname{example}{Example}{Examples}
\crefname{assumption}{Assumption}{Assumptions}
\crefname{hypothesis}{Hypothesis}{Hypotheses}
\crefname{fact}{Fact}{Facts}
\crefname{remark}{Remark}{Remarks}
\crefname{algocf}{Algorithm}{Algorithms}

\let\oldbreve\u

\renewcommand{\u}{\mathbf{u}}
\renewcommand{\v}{\mathbf{v}}

\renewcommand{\b}{\mathbf{b}}
\newcommand{\y}{\mathbf{y}}
\renewcommand{\d}{\mathbf{d}}
\renewcommand{\r}{\mathbf{r}}

\SetKwComment{RemarkComment}{\(\triangleright\)\ }{}

\newenvironment{rightequation}
  {\begin{equation}}
  {\end{equation}}

\newenvironment{keywords}
  {\par\smallskip\noindent\textbf{Keywords:}\enspace\ignorespaces}
  {\par\smallskip\ignorespacesafterend}
\newenvironment{AMS}
  {\par\smallskip\noindent\textbf{Mathematics Subject Classification (2020):}\enspace\ignorespaces}
  {\par\smallskip\ignorespacesafterend}

\title{A nonmonotone globalization framework for Anderson
acceleration for contractive and nonexpansive fixed point problems%
}

\author{
Wei Bian\thanks{\raggedright Corresponding author. School of Mathematics,
Harbin Institute of Technology, Harbin, China
(\href{mailto:bianweilvse520@163.com}{\texttt{bianweilvse520@163.com}}).}
\and
Zekai Li\thanks{\raggedright School of Mathematics,
Harbin Institute of Technology, Harbin, China
(\href{mailto:zeeklc_op@163.com}{\texttt{zeeklc\_op@163.com}}).}
}
\date{}

\begin{document}
\maketitle

% ==================== 从这里放入原稿内容 ====================
% 下面的英文仅为占位内容；请用原稿摘要、关键词和正文替换。
% 不需要再次粘贴原稿的 documentclass、导言区或 begin{document}。

\begin{abstract}
Anderson acceleration (AA) is an effective technique for accelerating fixed point iterations, but it generally lacks global convergence guarantees. We propose a nonmonotone globalized Anderson acceleration framework that retains the local acceleration of AA while ensuring global convergence. For contractive mappings, 
without requiring prior knowledge of the contraction factor,
we prove that the proposed method reduces exactly to pure AA after finitely many iterations and enjoys global $r$-linear convergence for memory size $m\geq1$ and global residual $q$-linear convergence for $m=1$, with convergence factors no greater than the contraction factor of the underlying fixed point mapping. For nonexpansive mappings, we prove that the fixed point residuals converge globally to zero. A class-agnostic parameter selection rule is further developed to ensure these convergence properties in both settings. Numerical experiments demonstrate the robustness and efficiency of the proposed method.
\end{abstract}

\begin{keywords}
Anderson acceleration, nonmonotone globalization, fixed point problem, linear convergence, nonexpansive mapping.
\end{keywords}

% 如需 MSC 分类，可原样粘贴原稿的 
\begin{AMS}
65H10, 65B05, 65K10
\end{AMS}

\section{Introduction}
In this paper, we consider the following fixed point problem: 
\begin{equation*}
    \mbox{Find} \,\,\u \in \mathbb{R}^n \,\, \mbox{such that} \,\, \u = G(\u),
\end{equation*}
where $G:\mathbb{R}^n \to \mathbb{R}^n$ is a Lipschitz continuous mapping and we denote the residual mapping of $G$ by $F(\u):=G(\u)-\u$. 
% {\color{red}Throughout this paper, $\|\cdot\|$ denotes the Euclidean norm.}
Fixed point problems arise naturally in a wide range of scientific and engineering applications, including nonlinear equations, optimization problems, and variational inequalities. Many widely used numerical methods can be formulated within the fixed point framework. For instance, gradient descent and proximal gradient methods, Douglas--Rachford splitting, and various nonlinear PDE solvers can all be interpreted as fixed point iterations associated with suitably defined mappings.
A basic approach to solving such problems is the Picard iteration:
\begin{equation*}
  \u_{k+1}=G(\u_k).
\end{equation*}
When $G$ is contractive with factor $c \in (0,1)$, the Banach fixed point theorem guarantees that the Picard iteration converges globally $q$-linearly to the unique fixed point. For nonexpansive mappings, one commonly employs suitably averaged iterations, such as the Krasnosel'ski\oldbreve{\i}--Mann iteration:
\begin{equation*}
    \u_{k+1}=(1-\lambda_k)\u_k+\lambda_k G(\u_k),
\end{equation*}
where $\lambda_k\in(0,1]$. Despite their simplicity and robustness, these basic fixed point methods often converge prohibitively slowly in practice.

Anderson acceleration (AA) \cite{Anderson, anderson2019comments} is a broadly applicable approach for accelerating fixed point iterations, originally introduced by Anderson to solve nonlinear integral equations. Owing to its effectiveness, AA has been widely applied in many fields, including self-consistent field iterations in electronic structure calculations \cite{banerjee2016periodic, chupin2021convergence, pulay1980convergence, walker2011anderson}, tensor decomposition \cite{sterck2012nonlinear, sterck2021asymptotic}, geometry optimization and physics simulation problems \cite{peng2018anderson}, among others \cite{an2017anderson, both2019anderson, lipnikov2013anderson, mai2020anderson, matveev2018anderson, pollock2019anderson, pratapa2016anderson}. At each iteration, let $m_k=\min\{m, k\}$. AA forms the new iterate by an affine combination of the most recent $m_k+1$ iterates, with the coefficients chosen to minimize the norm of the corresponding affine combination of residuals.
More precisely, pure AA computes $\u_{k+1}$ as follows:
\begin{equation}\label{eq_pureAA}
    \u_{k+1}=\sum\nolimits_{j=0}^{m_k}\alpha_j^kG(\u_{k-m_k+j}),
\end{equation}
where the coefficients \(\boldsymbol{\alpha}^k=(\alpha_0^k,\ldots,\alpha_{m_k}^k)^\top\) are obtained by solving
\begin{equation}\label{eq_subprob}
\min_{\boldsymbol{\alpha}\in\mathbb{R}^{m_k+1}}
\left\|
\sum\nolimits_{j=0}^{m_k}\alpha_jF(\u_{k-m_k+j})
\right\|
\quad\text{s.t.}\quad
\sum\nolimits_{j=0}^{m_k}\alpha_j=1.
\end{equation}

From a theoretical perspective, AA is closely related to multisecant quasi-Newton methods \cite{fang2009two} and, for linear problems, to GMRES \cite{walker2011anderson}. These features, together with its favorable numerical performance, have led to its successful application in a wide range of scientific computing problems. Despite its practical success, the convergence behavior of AA is considerably less well understood than that of the underlying fixed point iteration. Toth and Kelley \cite{TothKelley2015} established the first local convergence results for AA, proving local $r$-linear convergence for Lipschitz continuously differentiable contractive mappings under bounded Anderson coefficients. Chen and Kelley \cite{chen2019convergence} further relaxed the differentiability condition to continuous differentiability and obtained the same convergence results. 
% Wei et al. \cite{wei2022class} proposed a novel short-term recurrence variant of AA to reduce the memory overhead and established its local linear convergence for contractive mappings that are Lipschitz continuously differentiable. 
Wei et al. \cite{wei2022class} established local linear convergence of a short-term recurrence AA for smooth contractions.
Bian et al. \cite{bian2022anderson, bian2021anderson} further extended the theoretical results to a class of nonsmooth mappings and established analogous $r$-linear convergence results. 
Evans et al. \cite{evans2020proof} developed a one-step analysis showing AA can improve convergence rates for contractive mappings through a gain factor. Pollock and Rebholz \cite{pollock2021anderson} extended this analysis to noncontractive mappings, clarifying the role of higher-order perturbation terms.
% Evans et al. \cite{evans2020proof} proposed a one-step analysis of AA for contractive mappings, demonstrating that AA can actually improve the convergence rate with a gain factor. Pollock and Rebholz \cite{pollock2021anderson} further extended this gain factor analysis to a class of noncontractive mappings and clarified the balance between the acceleration gain and the higher-order perturbation terms.
De Sterck and He \cite{sterck2021asymptotic,de2022linear} analyzed asymptotic convergence factors of AA, while Krzysik et al. \cite{krzysik2025asymptotic} established improved factors for restarted AA(1) on symmetric and skew-symmetric linear mappings.

A major technical challenge associated with the original AA is that the method may become unstable or stagnate in practice. Existing stabilization and globalization strategies mainly follow two directions. The first employs safeguarding or restarting mechanisms and falls back to a globally convergent base iteration when the AA step is deemed unreliable. The second regularizes, constrains, or filters the least-squares subproblem \eqref{eq_subprob} to control the Anderson coefficients and improve its numerical conditioning. Chen and Kelley \cite{chen2019convergence} proposed the EDIIS method, which constrains the Anderson coefficients to be nonnegative and guarantees global $r$-linear convergence for contractive mappings. Zhang et al. \cite{zhang2020globally} developed a globally convergent Type-I AA method for nonsmooth nonexpansive mappings by combining safeguarding, Powell-type regularization, and restarting. Fu et al. \cite{fu2020anderson} incorporated regularized AA into Douglas--Rachford splitting and employed a residual-based safeguard to recover the convergence guarantees of the underlying splitting method. Although these safeguarded schemes guarantee global convergence for nonexpansive mappings, their analyses remain qualitative and yield no explicit convergence factors.
Mai and Johansson \cite{mai2020anderson} globalized AA for proximal gradient methods by accepting an accelerated point only when it satisfies a sufficient objective-decrease condition. Ouyang et al. \cite{ouyang2024descent} further developed a nonmonotone function value-based globalization for restarted AA applied to gradient methods. These strategies, however, rely on the availability and descent properties of an objective function and are therefore not directly applicable to general fixed point problems. Ouyang et al. \cite{ouyang2023nonmonotone} proposed a nonmonotone globalization framework based on adaptive regularization and the ratio between the actual and predicted residual reductions. However, in the contractive setting, their theoretical guarantees require prior knowledge of the contraction factor of the fixed point mapping, which is generally unavailable in practice.
In addition, regularization, restarting, and filtering strategies have been developed to control the Anderson coefficients and improve the conditioning of the least-squares subproblems \cite{pollock2023filtering, scieur2020regularized}.

The main contribution of this paper is the development of a globalized Anderson acceleration framework that combines global convergence guarantees with local acceleration. Moreover, the proposed framework requires little problem-specific information or parameter tuning and no prior knowledge of the contraction factor, thereby enhancing its practical applicability. For contractive mappings, we prove that the proposed algorithm reduces exactly to pure AA after finitely many iterations. We further establish global residual and iterate $r$-linear convergence for $m \ge 1$, and global residual $q$-linear convergence for $m=1$, with convergence factors no greater than the contraction factor of the fixed point mapping. Compared with the method proposed in~\cite{ouyang2023nonmonotone}, our algorithm has two advantages: it requires no prior knowledge of the contraction factor and eventually reduces to pure AA rather than regularized AA. For nonexpansive mappings, we establish the global convergence of the residual sequence to zero, complementing the global convergence results in~\cite{fu2020anderson, zhang2020globally}. 
Beyond this global residual convergence result, our framework retains linear convergence in the contractive setting without introducing a regularization term into the subproblem~\eqref{eq_subprob}.
Finally, we construct a class-agnostic parameter selection rule that guarantees the required convergence properties in both the contractive and nonexpansive settings.
These results demonstrate the generality of the proposed framework across both contractive and nonexpansive settings.

This work is organized as follows. In section \ref{sec2}, we present the proposed globalized Anderson acceleration algorithm and establish its convergence properties. Specifically, in subsection \ref{sec2.1}, we establish the global $r$-linear convergence of the proposed algorithm for contractive mappings with $m\geq1$, while subsection \ref{sec2.2} strengthens this result to global residual $q$-linear convergence for $m=1$. In subsection \ref{sec2.3}, we show that the residual sequence converges globally to zero for nonexpansive mappings. In subsection \ref{sec2.4}, we propose a class-agnostic strategy that ensures global residual convergence for nonexpansive mappings while preserving linear convergence in the contractive case. Finally, in section \ref{sec3}, we present numerical experiments to validate the effectiveness of the proposed algorithm, illustrating its globally convergent behavior, in contrast to conventional Anderson acceleration, as well as its robustness and competitive numerical performance relative to other acceleration methods.
% superiority over other acceleration methods.
\section{Global Anderson Acceleration Algorithm}\label{sec2}
This section provides the details of the proposed
global Anderson acceleration algorithm, together with the necessary preliminary concepts, assumptions and main theoretical results. 
We refer to the proposed algorithm as the Global Anderson($m$) algorithm, abbreviated as GlobAA($m$), and present it in Algorithm~\ref{glo-Anderson}.
We establish its global $r$-linear convergence for contractive mappings in \cref{sec2.1}, global residual $q$-linear convergence for $m=1$ in \cref{sec2.2}, and global residual convergence for nonexpansive mappings in \cref{sec2.3}. Moreover, we give a class-agnostic strategy for contractive and nonexpansive mappings in \cref{sec2.4}.

\begin{algorithm2e}[htbp]
\caption{Global Anderson($m$) Algorithm }\label{glo-Anderson}
\DontPrintSemicolon
\KwIn{$\u_0\in \mathbb{R}^n$, a positive integer $m$, and an update strategy $\mathcal{S}$.}
Given $\lambda \in (0,1]$; set $k =0 $ and $ \kappa = 0$. \\
\While{$\|F(\u_k)\|>0$}{
  choose $m_k = \min\{m,\; k-\kappa\}$;\\
  set ${G}_k = {G}(\u_k)$ and ${F}_k = {G}_k-\u_k$; \\
  choose \(\gamma_k\in(0,1)\) by strategy $\mathcal{S}$; \RemarkComment*[r]{\bf{\texttt{see Remark~\ref{remark1}}}}
  solve for $\{\alpha_j^k: j=0,\ldots,m_k\}$ by
    \begin{equation}\label{eq_coff}
      \min \left\|\sum\nolimits_{j=0}^{m_k}\alpha_j^k {F}_{k-m_k+j}\right\|
      \quad \text{s.t.}\quad \sum\nolimits_{j=0}^{m_k}\alpha_j^k = 1 ;
    \end{equation} 
    set
    % \vspace{-1mm}
    \begin{equation}\label{eq_AA}
        \u_{\rm AA}^{k+1} = \sum\nolimits_{j=0}^{m_k}\alpha_j^k\, {G}_{k-m_k+j};
    \end{equation}
    
    \uIf{ 
    \begin{rightequation}\label{global-con}
        \|{F}(\u_{\rm AA}^{k+1})\|\le \gamma_k \max_{0\leq j\leq m_k} \|F_{k-m_k+j}\| \tag{\textbf{\textsf{AC}}}
    \end{rightequation}
    }{
    \tcp{\bf{\texttt{Anderson step}}}
    set
    \begin{rightequation}\label{A}
        \u_{k+1} = \u_{\rm AA}^{k+1}; \tag{\textbf{\textsf{A}}}
    \end{rightequation}
  }\Else{
    \tcp{\bf{\texttt{Krasnosel'ski\oldbreve{\i}--Mann step}}}
    set
    \begin{rightequation}\label{K}
        \u_{k+1} = (1-\lambda) \u_k + \lambda{G}_k,\, \kappa = k+1; \tag{\textbf{\textsf{K}}}
    \end{rightequation}
  }
  $k=k+1.$
}
\end{algorithm2e}

\begin{remark}\label{remark1}
    Different assumptions are imposed on $\gamma_k$ depending on the properties of the underlying fixed point mapping $G$, as specified in the subsequent subsections.
    Parameter $\kappa$ primarily controls the memory parameter $m_k$, ensuring that $m_k$ is reset after each Krasnosel'ski\oldbreve{\i}--Mann step. Moreover, for contractive mappings, we allow $\lambda\in(0,1]$, whereas for nonexpansive mappings, we require $\lambda\in(0,1)$.
\end{remark}

It is easy to see that, at each iteration, GlobAA($m$) executes exactly one of the two processes \eqref{A} and \eqref{K}. We denote by $\mathbb{N}_{\ref{A}}$ and $\mathbb{N}_{\ref{K}}$ the sets of iteration indices at which processes \eqref{A} and \eqref{K} are performed, respectively.

In what follows, without loss of generality, we assume that the stopping criterion is never satisfied; otherwise, the algorithm would find a fixed point of $G$ after finitely many iterations. Therefore, for the sequence generated by the proposed algorithm, we assume that $F(\u_k) \neq 0$, $\forall k\geq0$.

\subsection{Global convergence for contractive mappings} In this section, we focus on the study of the proposed GlobAA($m$) for contractive mappings, where we will prove its global linear convergence by imposing some assumptions on $\gamma_k$.

\subsubsection{Global $r$-linear convergence for \texorpdfstring{$m \geq 1$}{m>=1}}\label{sec2.1}
We introduce our main assumptions on the mapping $G$ and emphasize that the analysis in this subsection does not rely on the optimality conditions satisfied by ${\alpha_j^k}$ in solving problem~\eqref{eq_coff}. Instead, it only requires the following assumptions on ${\alpha_j^k}$, which coincide with those in \cite{chen2019convergence, TothKelley2015}.
\begin{assumption}\label{anderson-ass1}
    The function $G:\mathbb{R}^n \to \mathbb{R}^n$ satisfies the following conditions:
    \begin{itemize}
        \item[(i)] $G$ is Lipschitz continuous with constant $c<1$;
        \item[(ii)] $G$ is differentiable at its fixed point $\u^*$.
    \end{itemize}
\end{assumption}
\begin{assumption}\label{anderson-ass2}
The coefficients $\{\alpha_j^k\}_{j=0}^{m_k}$ satisfy the following conditions for all $k\geq1$:
    \begin{itemize}
    \item[(i)] $\|\sum\nolimits_{j=0}^{m_k}\alpha_j^k{{F}}_{k-m_k+j}\| \leq \|{F}_k\|$;
    \item[(ii)] $\sum\nolimits_{j=0}^{m_k}\alpha_j^k=1$;
    \item[(iii)] there exists a constant $M_{\alpha}\geq 1$ such that $\sum_{j=0}^{m_k}|\alpha_j^k|\leq M_{\alpha}$ holds for all $k\geq1$.
\end{itemize}
\end{assumption}
\begin{remark}
    The first two conditions in Assumption~\ref{anderson-ass2} follow naturally from subproblem \eqref{eq_coff}, while condition (iii) is standard in AA convergence analysis \cite{chen2019convergence, TothKelley2015}. Our theoretical analysis is developed for the unregularized Anderson subproblem, with condition (iii) imposed as an analytical assumption. In practice, suitable regularization can promote condition (iii) and improve numerical stability by controlling the coefficients; we therefore consider it as an optional stabilization in the experiments.
\end{remark}

We next impose a condition on the parameters $\{\gamma_k\}$ in the update strategy $\mathcal{S}$, which plays a key role in ensuring the linear convergence and the eventual reduction of the proposed algorithm to the pure AA.
\begin{assumption}\label{ass3}
    The sequence $\{\gamma_k\}$ selected by the update strategy $\mathcal{S}$ satisfies the following conditions:
    \begin{itemize}
        \item[(i)] $\gamma_k\in(0,1)$ for all $k\geq 0$;
        \item[(ii)] $\gamma_k$ is non-decreasing and $\gamma_k \to 1$ as $k \to \infty$;
        \item[(iii)] $\sum_{k=0}^{\infty} (1-\gamma_k) = \infty$.
    \end{itemize}
\end{assumption}
\begin{remark}
 Conditions (i)-(iii) are mild and can be easily satisfied by many update strategies. For example, we can simply set $\gamma_k = 1 - \left(\frac{1}{k+\nu+1}\right)^p$ with $p\in(0,1]$ and $\nu \in \mathbb{N}_{+}$, which is a common choice for satisfying Assumption \ref{ass3}.
\end{remark}

\begin{theorem}
    \label{second}%
    Suppose that Assumptions~\ref{anderson-ass1}-(i) and \ref{ass3} hold, and let $\{\u_k\}$ be the sequence generated by GlobAA($m$). Then $\|F(\u_k)\| \rightarrow 0$ and $\{\u_k\}$ converges to $\u^*$ as $k\rightarrow \infty$. Moreover, if Assumptions~\ref{anderson-ass1}-(ii) and \ref{anderson-ass2} hold, there exists a $K_0\in \mathbb{N}$ such that the condition \eqref{global-con} is satisfied for every $k\geq K_0$, which implies that the iterates generated by GlobAA($m$) coincide with those generated by pure AA thereafter.
\end{theorem}

\begin{proof}
  We divide the proof into the following four steps.
  
  \textbf{Step 1:} Prove $\|{F}_{k}\| \rightarrow 0$ as $k \rightarrow \infty$ by considering the following two cases.

  {{- Case 1}} ($|\mathbb{N}_{\ref{A}}| < \infty$): This means that there are only finitely many Anderson steps in the iterations and we can assume that the last Anderson step occurs at iteration $K_{\max}$ (if $\mathbb{N}_{\ref{A}}=\emptyset$, we can simply set $K_{\max}=0$). When $k\in\mathbb{N}_{\ref{K}}$, we have $\u_{k+1} = (1-\lambda)\u_k + \lambda G_k$, and then 
        \begin{equation}\label{eq2.3}
            \begin{aligned}
            \|F_{k+1}\| &= \|G_{k+1}-\u_{k+1}\| = \|G_{k+1}-G_k+ (1-\lambda)(G_k-\u_k)\|\\
            &\leq c \|\u_{k+1} - \u_k\| + (1-\lambda)\|F_k\| 
            = (c \lambda + (1-\lambda))\|F_k\|. \\
            \end{aligned}
        \end{equation}
        Since the above inequality holds for $k > K_{\max}$ and $c \lambda + (1-\lambda)<1$, it follows that $\|{F}_{k}\| \rightarrow 0$ as $k \rightarrow \infty$.
        
  {{- Case 2}} ($|\mathbb{N}_{\ref{A}}| = \infty$): To simplify the notation, we denote
        \begin{equation*}
            W_k := \max_{0\leq j \leq m_k}\|{F}_{k-m_k+j}\|.
        \end{equation*}
        If $k \in \mathbb{N}_{\ref{A}}$, then $\|{F}_{k+1}\|\leq \gamma_k W_k$. Otherwise, by \eqref{eq2.3}, we have $\|{F}_{k+1}\|\leq (c \lambda + (1-\lambda)) \|{F}_k\|\leq (c \lambda + (1-\lambda)) W_k$. Denote $\eta_k :=\max\{\gamma_k, c \lambda + (1-\lambda)\}<1$. Then $\eta_k$ is non-decreasing. Hence, we obtain
        \begin{equation}\label{eq2.4-1}
            \|{F}_{k+1}\| \leq \eta_k W_k, \quad \forall k \geq 0.
        \end{equation}
        By the definition of $W_{k+1}$, we have $W_{k+1}\leq \max\{W_k, \|{F}_{k+1}\|\}\leq W_k$. Therefore, $\{W_k\}$ is a non-increasing sequence. 

        Since $\{\eta_k\}$ is non-decreasing and $\{W_k\}$ is non-increasing, \eqref{eq2.4-1} gives $\|{F}_{k+i+1}\|\leq \eta_{k+i} W_{k+i}\leq \eta_{k+m} W_{k+i} \leq \eta_{k+m} W_k$ {for $0\leq i \leq m$}. Recalling $\lim_{k\rightarrow\infty}\gamma_k=1$, there exists a $K_1$ such that $\gamma_k \geq c \lambda + (1-\lambda)$ for all $k\geq K_1$, which implies $\eta_k = \gamma_k$ for $k \geq K_1$. Thus, we have $W_{k+m+1} \leq \gamma_{k+m} W_k$ for all $k \geq K_1$. Therefore, 
        \begin{equation}\label{eq_prod}
            W_{K_1+l(m+1)} \leq W_{K_1}\prod_{j=0}^{l-1} \gamma_{K_1+j(m+1)+m}.
        \end{equation}
        Due to Assumption~\ref{ass3}-(ii) and (iii), we have $\sum_{j=0}^{\infty}(1- \gamma_{K_1+j(m+1)+m}) = +\infty$, which implies that $\prod_{j=0}^{\infty} \gamma_{K_1+j(m+1)+m} = 0$. Hence, \eqref{eq_prod} and the monotonicity of $\{W_k\}$ imply that $W_k\rightarrow 0$, and therefore $\|{F}_k\| \rightarrow 0$ as $k \rightarrow \infty$.

  \textbf{Step 2:} Convergence of $\{\u_k\}$. Since $G$ is a contraction mapping on $\mathbb{R}^n$ with contraction factor $c\in(0,1)$, we have
        \begin{equation}\label{eq_relationship}
            (1-c)\|\u-\u^*\|\leq \|F(\u)\|\leq (1+c)\|\u-\u^*\|.
        \end{equation}
        By $\|\u_k-\u^*\|\leq \frac{1}{1-c}\|F_k\|$, it follows that $\|\u_k-\u^*\|\to 0$ by $\|F_k\|\to 0$ as $k\to\infty$.

  \textbf{Step 3:} Estimation of $\|F(\u_{\rm AA}^{k+1})\|$.
        First, we have
        \begin{equation}\label{eq2.5-1}
        \begin{aligned}
            &\|F(\u_{\rm AA}^{k+1})\| = \|G(\u_{\rm AA}^{k+1}) - \u_{\rm AA}^{k+1}\| \\
            & \leq \left\|G(\u_{\rm AA}^{k+1})-G\left(\sum\nolimits_{j=0}^{m_k}\alpha_j^k \u_{k-m_k+j}\right)\right\|+\left\|G\left(\sum\nolimits_{j=0}^{m_k}\alpha_j^k \u_{k-m_k+j}\right) -  \u_{\rm AA}^{k+1}\right\| \\
            &\leq c \left\|\u_{\rm AA}^{k+1} - \sum\nolimits_{j=0}^{m_k} \alpha_j^k \u_{k-m_k+j} \right\|+\left\|G\left(\sum\nolimits_{j=0}^{m_k}\alpha_j^k \u_{k-m_k+j}\right) -  \u_{\rm AA}^{k+1}\right\| \\
            &\leq c \left\|\sum\nolimits_{j=0}^{m_k}\alpha_j^k\left(G(\u_{k-m_k+j})-\u_{k-m_k+j}\right) \right\| + \left\|G\left(\sum\nolimits_{j=0}^{m_k}\alpha_j^k \u_{k-m_k+j}\right) -  \u_{\rm AA}^{k+1}\right\| \\  
            &\leq c\|F_k\|+ \left\|G\left(\sum\nolimits_{j=0}^{m_k}\alpha_j^k \u_{k-m_k+j}\right) -  \u_{\rm AA}^{k+1}\right\|,
        \end{aligned}
        \end{equation}        
        where the last inequality follows from Assumption \ref{anderson-ass2}-(i).
        
        Next, we continue to estimate the second term on the right-hand side of the above inequality.
        Due to \eqref{eq_relationship}, we have 
        \begin{equation}\label{eq2.6-1}
        \begin{aligned}
            &\left \|\sum\nolimits_{j=0}^{m_k}\alpha_j^k \u_{k-m_k+j} - \u^*\right \| = \left\|\sum\nolimits_{j=0}^{m_k}\alpha_j^k\left(\u_{k-m_k+j} - \u^* \right) \right\| \\ &
            \leq \sum\nolimits_{j=0}^{m_k}|\alpha_j^k|\max_{0\leq j\leq m_k}\left\|\u_{k-m_k+j} - \u^*  \right\|\leq \frac{M_{\alpha}}{1-c}\max_{0\leq j\leq m_k}\|F_{k-m_k+j}\|,
        \end{aligned}
        \end{equation}
        by step 1, which shows that $\sum\nolimits_{j=0}^{m_k}\alpha_j^k \u_{k-m_k+j} \rightarrow \u^*$ as $k \rightarrow \infty$. %The above 
        Since $G$ is differentiable at $\u^*$, we have $G(\u) = G(\u^*) + G'(\u^*)(\u-\u^*) + o(\|\u-\u^*\|)$ as $\u\rightarrow\u^*$. Then, by $\u_k\rightarrow\u^*$ as $k\rightarrow\infty$ proved in step 2, there exists $K_2\in \mathbb{N}$ such that for all $k\geq K_2$, we have
        \begin{equation}\label{eq_AA_estimate}
        \begin{aligned}
            &\left\|G\left(\sum\nolimits_{j=0}^{m_k}\alpha_j^k \u_{k-m_k+j}\right) -  \u_{\rm AA}^{k+1}\right\| \\=&\left\|G\left(\sum\nolimits_{j=0}^{m_k}\alpha_j^k \u_{k-m_k+j}\right) - G(\u^*) + G(\u^*) - \u_{\rm AA}^{k+1}\right\| \\
            =& \left\|G'(\u^*)\left(\sum\nolimits_{j=0}^{m_k}\alpha_j^k \u_{k-m_k+j} - \u^*\right)+ G(\u^*) - \u_{\rm AA}^{k+1}\right\|+o\left(\max_{0\leq j\leq m_k}\|F_{k-m_k+j}\|\right) \\
            =& o\left(\max_{0\leq j\leq m_k}\|F_{k-m_k+j}\|\right), 
        \end{aligned}
        \end{equation}
        where the second equality follows from the differentiability of $G$ at $\u^*$ and \eqref{eq2.6-1}, and the third equality holds by using \eqref{eq_AA} and \eqref{eq_relationship} together with the differentiability of $G$ at $\u^*$ to obtain
        \begin{equation*}
        \begin{aligned}
    &\left\|G'(\u^*)\left(\sum_{j=0}^{m_k}\alpha_j^k \u_{k-m_k+j} - \u^*\right)+ G(\u^*) - \u_{\rm AA}^{k+1}\right\|\\
                    =& \left\|\sum_{j=0}^{m_k} \alpha_j^k \left[G'(\u^*)(\u_{k-m_k+j} - \u^*) + G(\u^*) - G(\u_{k-m_k+j})\right]\right\|\\
            \leq& M_{\alpha}\max_{0\leq j\leq m_k}o(\|\u_{k-m_k+j} - \u^*\|)=o\left(\max_{0\leq j\leq m_k}\|F_{k-m_k+j}\|\right).
        \end{aligned}
        \end{equation*}
        Combining \eqref{eq2.5-1} and \eqref{eq_AA_estimate}, and using the definition of $W_k$, we obtain, for all $k\geq K_2$, that
        \begin{equation}\label{F_AA}
            \|F(\u_{\rm AA}^{k+1})\| \leq c\|F_k\| + {o}(W_k).
        \end{equation}
        
        \textbf{Step 4:} Verification of condition \eqref{global-con}. 
        Let $\delta>0$ be given satisfying $c+\delta<1$.
        Due to Assumption~\ref{ass3}-(ii), there exists a $K_3\geq K_2$ such that $\gamma_k \geq c+\delta$ for all $k\geq K_3$. 
        By \eqref{F_AA}, there exists a $K_4\geq K_3$ such that for all $k\geq K_4$, we have 
        \begin{equation*}
            \|F(\u_{\rm AA}^{k+1})\| \leq c\|F_k\| + \delta W_k \leq (c+\delta) W_k \leq \gamma_k W_k.
        \end{equation*}
        Therefore, we can set $K_0 = K_4$, and for all $k\geq K_0$, condition \eqref{global-con} is satisfied, which means that every iteration for $k\geq K_0$ is an Anderson step. In other words, GlobAA($m$) reduces to the pure AA after finitely many iterations.
\end{proof}

\begin{remark}
  Theorem~\ref{second} remains valid if conditions (ii) and (iii) in Assumption~\ref{ass3} are replaced by the following weaker conditions: 
    \begin{itemize}
        \item [(ii’)] $\gamma_k \to 1$ as $k \to \infty$;
         \item [(iii’)] there exists a $k_0\in \mathbb{N}$ such that $\sum_{j=0}^{\infty} \left(1-\max_{0\leq i \leq m}\gamma_{k_0+j(m+1)+i}\right) = +\infty$.
    \end{itemize}
      In this case, we can define $\hat{\eta}_k$ in Step 1 as $\hat{\eta}_k:=\max_{0\leq i\leq m}\eta_{k+i}$, where ${\eta_k}$ is defined as in Theorem~\ref{second}. Then, \eqref{eq2.4-1} together with the monotonicity of ${W_k}$ yields $W_{k+m+1}\leq \hat{\eta}_k W_k$.
    Although these conditions are weaker, they are more technical. We therefore retain the simpler formulation stated in Assumption~\ref{ass3}.
\end{remark}
\begin{remark}
These results also help position GlobAA($m$) relative to several existing stabilization and globalization strategies for AA. Most existing approaches enhance the robustness of AA through different mechanisms, including controlling the Anderson coefficients \cite{chen2019convergence} and incorporating regularization or safeguarding strategies \cite{fu2020anderson, ouyang2023nonmonotone}. 
In particular, among the results in \cite{ouyang2023nonmonotone}, the analysis for contractive mappings is developed with prior knowledge of the contraction factor and employs regularization throughout the iteration.
GlobAA($m$), on the other hand, takes a complementary approach: it preserves the original Anderson subproblem without modifying or regularizing the Anderson coefficients, does not require prior knowledge of the contraction factor, and, as established in Theorem~\ref{second}, eventually reduces exactly to pure AA after finitely many iterations. Thus, GlobAA($m$) provides globalization while retaining the original AA mechanism once the iterates enter a regime in which pure AA can be safely employed. 
The numerical experiments in subsections~\ref{sec3.1} and \ref{sec3.2} further illustrate the practical implications of these differences. In the contractive examples considered here, the residual-based safeguarding strategy in \cite{fu2020anderson} may still accept Anderson steps that provide insufficient progress toward convergence. By contrast, GlobAA($m$) rejects such steps when necessary and exhibits robust convergence across different initial points and memory sizes.
\end{remark}
\begin{theorem}\label{third}
    Suppose that Assumptions~\ref{anderson-ass1}, \ref{anderson-ass2} and \ref{ass3} hold, and let $\{\u_k\}$ be the sequence generated by GlobAA($m$). Then 
    \begin{equation*}
    % \label{eq_r_factor}
        \limsup_{k \to \infty} \left(\frac{\|F(\u_{k})\|}{\|F(\u_0)\|}\right)^{{1}/{k}}\leq c, \quad \limsup_{k \to \infty} \left(\frac{\|\u_{k} - \u^*\|}{\|\u_{0}-\u^*\|}\right)^{{1}/{k}} \leq c.
    \end{equation*}
\end{theorem}
\begin{proof} 
    Let $\epsilon>0$ be given satisfying $c+\epsilon<1$. By Theorem~\ref{second}, $\u_{k+1}=\u_{\rm AA}^{k+1}$ for all $k\geq K_0$. 
    Set $q = {\epsilon}(c+\epsilon)^m$. By \eqref{F_AA}, there exists a $\bar{K} \geq K_0$ with $K_0$ given in Theorem \ref{second}, such that for all $k\geq \bar{K}$, we have
    \begin{equation}\label{eq_re}
        \|F_{k+1}\| \leq c \|F_k\| + q W_k. 
    \end{equation}
    Define $\Gamma:=(c+\epsilon)^{-\bar{K}}\max_{0\leq j \leq m_{\bar{K}}}\|{F}_{\bar{K}-m_{\bar{K}}+j}\|$. For $\bar{K}-m_{\bar{K}}\leq l \leq \bar{K}$, we have 
    \begin{equation}\label{eq_22}
    \begin{aligned}
        \|{F}_{l}\|&\leq\max_{0\leq j \leq m_{\bar{K}}}\|{F}_{\bar{K}-m_{\bar{K}}+j}\| =\left(\max_{0\leq j \leq m_{\bar{K}}}\|{F}_{\bar{K}-m_{\bar{K}}+j}\|\right)\cdot(c+\epsilon)^{-l}\cdot(c+\epsilon)^{l} \\
        &\leq \left(\max_{0\leq j \leq m_{\bar{K}}}\|{F}_{\bar{K}-m_{\bar{K}}+j}\|\right)\cdot(c+\epsilon)^{-\bar{K}}\cdot(c+\epsilon)^l 
        = \Gamma(c+\epsilon)^{l}.
    \end{aligned}
    \end{equation}
    We next claim that the estimate $\|{F}_{k}\|\leq \Gamma (c+\epsilon)^k$ holds for all $k\geq \bar{K}$. The base case $k = \bar{K}$ follows directly from \eqref{eq_22}. 
    We prove this claim by induction. Suppose that $\|{F}_{i}\|\leq \Gamma (c+\epsilon)^i$ holds for all $i = \bar{K}, \bar{K}+1, \dots, k$. Together with \eqref{eq_22}, this estimate also holds for $i = \bar{K}-m_{\bar{K}}, \dots, \bar{K}$. Moreover, since every iteration $i\geq \bar{K}$ is an Anderson step, we have $k-m_k \geq \bar{K}-m_{\bar{K}}$.
    By \eqref{eq_re}, we have
    \begin{equation*}
    \begin{aligned}
        \|F_{k+1}\| &\leq c \|F_k\| + q W_k \leq c \cdot \Gamma (c+\epsilon)^k + q \cdot \Gamma (c+\epsilon)^{k-m} \\
        &= \Gamma (c+\epsilon)^{k} \left(c + q(c+\epsilon)^{-m}\right) = \Gamma (c+\epsilon)^{k+1},
    \end{aligned}
    \end{equation*}
    where the second inequality follows from 
    \begin{equation*}
        W_k = \max_{0\leq j \leq m_{k}}\|{F}_{k-m_{k}+j}\|\leq \Gamma \max_{0\leq j \leq m_{k}}(c+\epsilon)^{{k- m_{k}}+j} \leq \Gamma (c+\epsilon)^{k-m}.
    \end{equation*}
    Therefore, we have $\|{F}_{k}\|\leq \Gamma (c+\epsilon)^k$ for all $k\geq \bar{K}$, which implies that 
    \begin{equation*}
        \limsup_{k \to \infty} \left(\frac{\|F(\u_{k})\|}{\|F(\u_0)\|}\right)^{\frac{1}{k}} \leq c+\epsilon.
    \end{equation*}
    Since $\epsilon>0$ is arbitrary, we have 
    \begin{equation*}
        \limsup_{k \to \infty} \left(\frac{\|F(\u_{k})\|}{\|F(\u_0)\|}\right)^{\frac{1}{k}} \leq c.
    \end{equation*}
    From \eqref{eq_relationship} and $\lim_{k\rightarrow\infty}\left(\frac{1+c}{1-c}\right)^{1/k}=1$, we can also obtain the $r$-linear convergence of $\|\u_k-\u^*\|$ with factor $c$, i.e.,
    \[
    \limsup_{k \to \infty} \left(\frac{\|\u_{k} - \u^*\|}{\|\u_{0}-\u^*\|}\right)^{{1}/{k}} \leq c.
    \]
\end{proof}

\subsubsection{Global residual $q$-linear convergence for \texorpdfstring{$m = 1$}{m=1}}\label{sec2.2}
In this subsection, we focus on the residual convergence analysis of GlobAA(1). For $m_k=1$, when $F_k \neq F_{k-1}$, the optimization problem in \eqref{eq_coff} has a closed form solution given by
\begin{equation}\label{eq_alpha}
    \alpha^k = \frac{ F_k^\top (F_k - F_{k-1})}{\|F_k - F_{k-1}\|^2}
\end{equation}
and then \eqref{eq_AA} is reduced to $\u_{\rm AA}^{k+1} = (1-\alpha^k)G_k + \alpha^k G_{k-1}$. As for the case when $F_k = F_{k-1}$, we can simply set $\alpha^k = 0$, which means that $\u_{\rm AA}^{k+1} = G_k$.

\begin{theorem}
    \label{fourth}%
    Suppose that Assumptions~\ref{anderson-ass1}, \ref{anderson-ass2} and \ref{ass3} hold, and let $\{\u_k\}$ be the sequence generated by GlobAA(1). Then, we have 
    \begin{equation*}
        \limsup_{k \to \infty} \frac{\|F(\u_{k+1})\|}{\|F(\u_k)\|}\leq c.
    \end{equation*}
\end{theorem}
\begin{proof}
    Let $\epsilon>0$ be given satisfying $c+2\epsilon<1$. By Theorem~\ref{second} and Theorem~\ref{third}, 
    there exist ${K_1}\in \mathbb{N}$ and $\Gamma>0$ such that every iteration $k\geq{K_1}$ is an Anderson step, $m_k=1$, and 
    \begin{equation}\label{eq_q_1}
        \|F_k\|\leq \Gamma (c+\epsilon)^k.
    \end{equation}
    First, we claim that there exist infinitely many $k \geq {K_1}$ such that $\|F_{k}\| \leq (c+2\epsilon)\|F_{k-1}\|$. Suppose, for contradiction, that the statement is not true. Then, there exists a $\tilde{K}\geq{K_1}$ such that for all $k \geq \tilde{K}$, we have $\|F_{k}\| > (c+2\epsilon)\|F_{k-1}\|$. Therefore, we have $\|F_{k}\| > (c+2\epsilon)^{k-\tilde{K}+1}\|F_{\tilde{K}-1}\|$ for all $k \geq \tilde{K}$. 
    Together with \eqref{eq_q_1}, we have $(c+2\epsilon)^{k-\tilde{K}+1}\|F_{\tilde{K}-1}\| < \|F_k\| \leq \Gamma (c+\epsilon)^k$, which implies that 
    \begin{equation*}
        \left(\frac{c+2\epsilon}{c+\epsilon}\right)^k < \frac{\Gamma (c+2\epsilon)^{\tilde{K}-1}} {\|F_{\tilde{K}-1}\|}.
    \end{equation*}
    Since $\frac{c+2\epsilon}{c+\epsilon}>1$ and the right-hand side is a constant, the above inequality cannot hold for sufficiently large $k$. Thus, there exist infinitely many $k \geq {K_1}$ such that 
    \begin{equation}\label{eq_q_2}
            \|F_{k}\| \leq (c+2\epsilon)\|F_{k-1}\|.
    \end{equation}
    Second, let $k\geq K_1$ be such that \eqref{eq_q_2} holds. Then we have $\|F_k\|\leq(c+2\epsilon)\|F_{k-1}\|$, which implies that $\|F_k-F_{k-1}\|\geq \|F_{k-1}\|-\|F_k\|\geq (1-c-2\epsilon)\|F_{k-1}\|$. Together with \eqref{eq_relationship} and \eqref{eq_alpha}, we have
    \begin{equation}\label{eq2.20}
        \begin{aligned}
            |\alpha^k|\|\u_{k-1}-\u^*\| \leq |\alpha^k|\frac{\|F_{k-1}\|}{1-c}\leq\frac{\|F_{k}\|}{(1-c-2\epsilon)(1-c)} 
        \end{aligned}
    \end{equation}
    and 
    \begin{equation}\label{eq2.21}
        \begin{aligned}
            |1-\alpha^k|\|\u_k-\u^*\| \leq |1-\alpha^k|\frac{\|F_k\|}{1-c} \leq \frac{\|F_k\|}{(1-c-2\epsilon)(1-c)}.
        \end{aligned}
    \end{equation}
  
    By \eqref{eq2.6-1} and step 2 of Theorem~\ref{second}, we have $(1-\alpha^k)\u_k + \alpha^k \u_{k-1}\rightarrow \u^*$ and $\u_k\rightarrow\u^*$ as $k \to \infty$. Similar to the analysis of \eqref{eq_AA_estimate}, and by \eqref{eq2.20}--\eqref{eq2.21}, there exists $K_2> K_1$ such that, for all $k\geq K_2$ satisfying \eqref{eq_q_2}, we have
    \begin{equation*}
    % \label{eq_AA1_estimate}
        \begin{aligned}
            &\left\|G\left((1-\alpha^k)\u_k + \alpha^k \u_{k-1}\right) -  \u_{\rm AA}^{k+1}\right\| \\
            = &\left\|G\left((1-\alpha^k)\u_k + \alpha^k \u_{k-1}\right) - G(\u^*) + G(\u^*) - ((1-\alpha^k)G_k + \alpha^k G_{k-1})\right\| \\
            \leq& o\left(|1-\alpha^k|\|\u_k-\u^*\|\right)+o\left(|\alpha^k|\|\u_{k-1}-\u^*\|\right)\\
            =&o\left(\|F_k\|\right).
        \end{aligned}
    \end{equation*}
    Then, by \eqref{eq2.5-1}, we have 
    \[
    \|F_{k+1}\|\leq c\|F_k\|+o\left(\|F_k\|\right)
    \]
    for all $k\geq {K_2}$ satisfying \eqref{eq_q_2}.
    Furthermore, there exists a $K_3\geq K_2$ such that, for every such $k\geq {K}_3$, we have $o\left(\|F_k\|\right) \leq 2\epsilon \|F_k\|$. 

    Choose an index $\hat{K}\geq K_3$ satisfying \eqref{eq_q_2}, whose existence follows from the above analysis. We prove that \eqref{eq_q_2} holds for all $k\geq \hat{K}$ by induction. The base case follows from the choice of $\hat{K}$. Suppose that \eqref{eq_q_2} holds for every $k = \hat{K}, \dots, K$ with $K\geq \hat{K}$. In particular, we have $\|F_K\|\leq(c+2\epsilon)\|F_{K-1}\|$. Then, we can estimate $\|F_{K+1}\|$ as follows:
    \begin{equation*}
        \|F_{K+1}\| \leq c \|F_K\| + o\left(\|F_K\|\right) \leq (c+2\epsilon) \|F_K\|.
    \end{equation*}
    Therefore, we have $\|F_{k+1}\| \leq (c+2\epsilon)\|F_k\|$ for all $k\geq \hat{K}$ by induction. By the arbitrariness of $\epsilon>0$, we have
    \[
    \limsup_{k \to \infty} \frac{\|F(\u_{k+1})\|}{\|F(\u_k)\|} \leq c.
    \]
\end{proof}

\subsection{Global residual convergence for nonexpansive mappings}\label{sec2.3}
This subsection is devoted to analyzing the global residual convergence of GlobAA($m$) for nonexpansive mappings, starting with several assumptions.
\begin{assumption}\label{anderson-global}
    The function $G:\mathbb{R}^n \to \mathbb{R}^n$ satisfies the following conditions:
    \begin{itemize}
        \item[(i)] $G$ is a nonexpansive mapping on $\mathbb{R}^n$, i.e.,
        \begin{equation*}
            \|G(\u)-G(\v)\|\leq \|\u-\v\|,\quad \forall \u, \v\in\mathbb{R}^n.
        \end{equation*}
        \item[(ii)] $G$ has a fixed point $\u^*$, i.e., $G(\u^*)=\u^*$.
    \end{itemize}
\end{assumption}

\begin{assumption}\label{anderson-global-gamma}
 Under the update strategy $\mathcal{S}$, the sequence $\{\gamma_k\}$ is generated according to the following conditions
     \begin{equation}\label{eq_new_theta_1}
        0<\gamma_k<1 \quad\mbox{and}\quad \gamma_k\leq1-\Theta\left({W_k}/{W_0}\right),
    \end{equation}
    where $W_k = \max_{0\leq j\leq m_k} \|F_{k-m_k+j}\|$, and $\Theta: [0,1]\rightarrow [0,1)$ is a non-decreasing function satisfying $\Theta(0)=0$ and $\Theta(t)>0$ for any $t>0$.
\end{assumption}

\begin{remark}
    It is easy to verify that $\gamma_k$ is well-defined since $W_k\leq W_0$ for $k\geq 0$.
    The conditions in Assumption~\ref{anderson-global-gamma} admit several simple choices of $\Theta$ and $\gamma_k$.
    For example, taking
$
\Theta(t)=\frac{t}{2(1+t)}$
and $\gamma_k=1-\frac{W_k}{2(W_0+W_k)}$
satisfies the required conditions. Alternatively, choosing
\[
\Theta(t)=
\begin{cases}
0, & t=0,\\
1-\xi, & t>0,
\end{cases}
\]
with $\xi\in(0,1)$, makes any sequence $\{\gamma_k\}$ with $\gamma_k\in(0,\xi]$ admissible. Moreover, another simple choice is $\Theta(t) = \frac{t}{2}$, for which one may take $\gamma_k = 1-\frac{W_k}{2W_0}$, and we prove that $0<\gamma_k<1$ in subsection~\ref{sec2.4}.

\end{remark}
\begin{theorem}
    \label{first}%
    Suppose that Assumptions~\ref{anderson-global}--\ref{anderson-global-gamma} hold and $\lambda\in (0,1)$. Then the sequence $\{\u_k\}$ generated by GlobAA($m$) satisfies $\|F(\u_k)\| \rightarrow 0$.
\end{theorem}
\begin{proof}
    We prove this theorem by considering the following two cases.

        {\bf{Case 1:}} $|\mathbb{N}_{\ref{A}}| < \infty$, which means that there are only finitely many Anderson steps in the iterations. Thus, there exists a $K_{\rm KM}$ such that $k\in\mathbb{N}_{\ref{K}}$ for all $k \geq K_{\rm KM}$. Hence,
        we have $\u_{k+1} = (1-\lambda)\u_k + \lambda G_k$ for any $k \geq K_{\rm KM}$, and then 
        \begin{equation}\label{eq_km}
            \begin{aligned}
            \|\u_{k+1}- &\u^*\|^2 = \|(1-\lambda)(\u_k-\u^*) + \lambda(G_k-G(\u^*))\|^2 \\
            &=(1-\lambda)\|\u_k-\u^*\|^2 +\lambda\|G_k-G(\u^*)\|^2-\lambda(1-\lambda)\|G_k-\u_k\|^2 \\
            &\leq \|\u_k-\u^*\|^2 - \lambda(1-\lambda)\|F_k\|^2,
            \end{aligned}
        \end{equation}
        where the second equality follows from $\|(1-\lambda){\bf{a}}+\lambda {\bf{b}}\|^2=(1-\lambda)\|{\bf{a}}\|^2+\lambda \|{\bf{b}}\|^2-\lambda(1-\lambda)\|{\bf{a}}-{\bf{b}}\|^2$ and the last inequality holds because $G$ is nonexpansive. For any $N \geq K_{\rm KM}$, summing \eqref{eq_km} over $k = K_{\rm KM}, \dots, N$, we obtain
        \[
        \sum_{k=K_{\rm KM}}^{N} \lambda(1-\lambda)\|F_k\|^2 \leq \|\u_{K_{\rm KM}}-\u^*\|^2 - \|\u_{N+1}-u^*\|^2 \leq \|\u_{K_{\rm KM}}-\u^*\|^2.
        \]
        Letting $N\rightarrow\infty$ and using $\lambda\in(0,1)$, we obtain $\sum_{k=K_{\rm KM}}^{\infty}\|F_k\|^2 < \infty$, which implies that $\|F_k\|\rightarrow 0$ as $k\rightarrow \infty$.

        {\bf{Case 2:}} $|\mathbb{N}_{\ref{A}}| = \infty$, which means that there are infinitely many Anderson steps in the iterations. 
        Let us denote
        \[
        W_k = \max_{0\leq j\leq m_k} \|F_{k-m_k+j}\|.
        \]
        We first show that the sequence $\{W_k\}$ is non-increasing, {which implies that $W_k\leq W_0$.} To this end, we consider the following two cases:
        
          - If $k\in\mathbb{N}_{\ref{A}}$, by condition \eqref{global-con}, we have 
          \begin{equation}\label{eq1}
            \|F_{k+1}\| \leq \gamma_k W_k \leq W_k.
          \end{equation} 
          
          - If $k\in\mathbb{N}_{\ref{K}}$, then we have $\u_{k+1} = (1-\lambda)\u_k + \lambda G_k$, which implies that 
          \begin{equation}\label{KM-residual}
            \begin{aligned}
              \|F_{k+1}\| &= \|G_{k+1} - \u_{k+1}\|\leq\|G_{k+1}-G_k\|+\|(1-\lambda)(G_k-\u_k)\| \\
              &\leq\|\u_{k+1}-\u_k\|+(1-\lambda)\|F_k\| = \lambda\|G_k-\u_k\|+(1-\lambda)\|F_k\| \\
              &= \|F_k\| \leq W_k.
            \end{aligned}
          \end{equation}
        Combining \eqref{eq1} and \eqref{KM-residual}, we conclude that $\|F_{k+1}\|\leq W_k$ for all $k\geq 0$. Then, $W_{k+1}\leq \max\{\|F_{k+1}\|, W_k\}\leq W_k$. This shows that $\{W_k\}$ is non-increasing, and then $\lim_{k\to\infty}W_k = W_{\infty}$.

        Second, we will claim the statement that
        \begin{equation}\label{eq_claim}
            W_{k+m+1}\leq \zeta_k W_k, \quad \forall k\in \mathbb{N}_{\ref{A}},
        \end{equation}
        where $\zeta_k:=\max\{\gamma_j: j \in \mathbb{N}_{\ref{A}}, k\leq j\leq k+m \}$.
        
        To do this, we will show by induction that $\|F_{p+1}\|\leq \zeta_k W_k$ for all $p=k,...,k+m$. 
        It is easy to verify that the estimate holds for $p=k$ by \eqref{eq1} and $\gamma_k \leq \zeta_k$. 
        Suppose that $\|F_{p+1}\|\leq \zeta_k W_k$ holds for all $p = k, \dots, q-1$ with some $k+1 \leq q\leq k+m$, and we will prove $\|F_{q+1}\|\leq \zeta_k W_k$.
    
        If $q\in \mathbb{N}_{\ref{A}}$, we have $\gamma_{q} \leq \zeta_k$. Since $\{W_k\}$ is non-increasing and $q \geq k$, we have $W_q\leq W_k$. Therefore,
        \begin{equation*}
        % \label{eq2}
            \|F_{q+1}\|\leq \gamma_{q} W_{q}\leq \zeta_k W_k.
        \end{equation*}
        If $q \in \mathbb{N}_{\ref{K}}$,  
        we have 
        \[
        \|F_{q+1}\|\leq \|F_{q}\|\leq \zeta_k W_k,
        \]
        where the last inequality follows from the induction hypothesis. This closes the induction and thereby establishes the proof of statement \eqref{eq_claim}.

        Third, we will prove $W_{\infty} = 0$ by contradiction. Assume that $W_{\infty} > 0$. Since $W_k \geq W_{\infty}$ and $\Theta$ is non-decreasing, we have
        \begin{equation}\label{eq_theta_1}
        \Theta\left({W_k}/{W_0}\right)\geq \Theta\left({W_{\infty}}/{W_0}\right).
        \end{equation}
        Let $\delta := \Theta\left(W_{\infty}/{W_0}\right)\in(0,1)$. Then 
        \begin{equation}\label{eq_theta_2}
            \gamma_k\leq1-\Theta\left({W_k}/{W_0}\right)\leq1-\delta, \quad \forall k\geq0.
        \end{equation}
        Due to $|\mathbb{N}_{\ref{A}}| = \infty$, there exists a subsequence $\{l_i\}_{i=1}^{\infty} \subseteq \mathbb{N}_{\ref{A}}$ such that $l_{i+1} > l_i + m$. Since $\{W_k\}$ is non-increasing, by \eqref{eq_claim}, we have 
        \[
            W_{l_{i+1}}\leq W_{l_i+m+1}\leq \zeta_{l_i} W_{l_i}\leq (1-\delta)W_{l_i},
        \]
        which implies that $W_{l_{i}}\leq (1-\delta)^{i-1}W_{l_1}$. Therefore, $W_{l_i}\rightarrow0$ as $i\to\infty$, which contradicts $W_{\infty}>0$. In conclusion, we have $W_k\to 0$ and $\|F(\u_k)\|\to0$ as $k\to \infty$.
\end{proof}

\begin{remark}
    The preceding proof clarifies the role of Assumption~\ref{anderson-global-gamma}. If $|\mathbb{N}_{\ref{A}}|<\infty$, then after finitely many iterations the algorithm reduces to the KM iteration, and the residual convergence follows directly from the descent property of the KM steps. Hence, the essential role of Assumption~\ref{anderson-global-gamma} arises only when $|\mathbb{N}_{\ref{A}}|=\infty$. 
{In the contradiction argument, it suffices that for every $\epsilon\in(0,1)$, there exists a $\delta>0$ such that $\Theta(t)\geq\delta$ for $t\in[\epsilon,1]$. 
    Therefore, the conditions of $\Theta$ in Assumption~\ref{anderson-global-gamma} can be relaxed to  
    \begin{equation}\label{eq_new_theta_2}
        \inf_{t\in[\epsilon,1]}\Theta(t)>0, \quad\mbox{for any fixed $\epsilon\in(0,1)$}.
    \end{equation}}
\end{remark}
\subsection{A class-agnostic strategy for contractive and nonexpansive mappings}\label{sec2.4}
In this subsection, we construct a class-agnostic strategy $\mathcal{S}$, requiring no prior knowledge of the mapping class, that yields global $r$-linear convergence in the contractive case and global residual convergence in the nonexpansive case.
To this end, the sequence $\{\gamma_k\}$ generated by the update strategy $\mathcal{S}$ is given by
% An Anderson step with $m_k=0$ coincides with a Picard step, since in this case $\u_{k+1} = G(\u_k)$.
\begin{equation}\label{eq_new_gamma}
    \gamma_k = 1 - \max\left\{\frac{1}{k+2}, \frac{W_k}{2 W_0}\right\},
\end{equation}
{where $W_k = \max_{0\leq j\leq m_k} \|F_{k-m_k+j}\|$.}
\begin{theorem}
Let $\lambda\in (0,1)$, and suppose the sequence $\{\gamma_k\}$ in strategy $\mathcal{S}$ is generated by \eqref{eq_new_gamma}. Let $\{\u_k\}$ denote the sequence produced by GlobAA($m$). Then, the following conclusions hold:
\begin{itemize}
\item if Assumption~\ref{anderson-global} holds, then $\|F(\u_k)\| \rightarrow 0$;
\item if Assumption~2.12 holds, and $G$ is contractive with factor $c$ in a neighborhood $D$ of a fixed point $\u^*$, then $\u^*$ is the unique fixed point. Moreover, if $G$ is differentiable at $\u^*$, and Assumption~\ref{anderson-ass2} holds, then the conclusions of Theorems~\ref{second}, \ref{third} and \ref{fourth} remain valid.
\end{itemize}
\end{theorem}
\begin{proof}
First, we prove that $0<\gamma_k<1$ for all $k\geq0$, while simultaneously showing that $\{W_k\}$ is non-increasing and $\{\gamma_k\}$ is non-decreasing. 
Note that $\gamma_k\leq 1-\frac{1}{k+2}<1$ for all $k\geq0$. It therefore remains to show that $\gamma_k>0$ for all $k\geq0$. We prove this by induction, where the case $k=0$ is immediate by $\gamma_0=1/2$. Suppose that $\gamma_k>0$ for all $0\leq k \leq K$. Then,
due to \eqref{eq1} and \eqref{KM-residual}, we have $\|F_{K+1}\|\leq W_K$, which implies $W_{K+1}\leq \max\{\|F_{K+1}\|, W_K\}\leq W_K$ and hence $\gamma_{K+1}\geq\gamma_K>0$, which completes the induction. Therefore, $0<\gamma_k<1$ for all $k\geq 0$, and $\{W_k\}$ is non-increasing.
Due to $\frac{1}{k+3}\leq\frac{1}{k+2}$ and $W_{k+1}\leq W_k$, we have 
    \[
    \max\left\{\frac{1}{k+3}, \frac{W_{k+1}}{2W_0} \right\}\leq \max\left\{\frac{1}{k+2}, \frac{W_{k}}{2W_0} \right\},
    \]
    which implies that $\gamma_{k+1}\geq\gamma_k$. 

Second, we prove $W_k\to 0$ and $\|F(\u_k)\|\to0$ as $k\to \infty$. 
To verify this, by Theorem \ref{first}, it suffices to check that the $\{\gamma_k\}$ in \eqref{eq_new_gamma} satisfies Assumption \ref{anderson-global-gamma}, where we only need to check the second condition in \eqref{eq_new_theta_1}. By \eqref{eq_new_gamma}, it follows that 
\[
1 - \gamma_k =\max\left\{\frac{1}{k+2}, \frac{W_k}{2W_0}\right\}\geq \frac{W_k}{2W_0},
\]
which yields \eqref{eq_new_theta_1} with $\Theta(t)=\frac{1}{2}t$.

Let $\mathcal{B}(\rho):=\{\u:\|\u-\u^*\|\leq \rho\}\subset D$. For any $\u\notin\mathcal{B}(\rho)$, by the nonexpansiveness of $G$ and its local contractivity, we have
\[
\|\u-\u^*\|\le \|F(\u)\|+\|\u-\v\|+c\|\v-\u^*\|=\|F(\u)\|+\|\u-\u^*\|-(1-c)\rho,
\]
where $\v$ is the point on the segment between $\u^*$ and $\u$ such that $\|\v-\u^*\|=\rho$, and then $\|F(\u)\|\geq (1-c)\rho$. 
The local contractivity ensures uniqueness of the fixed point in $\mathcal{B}(\rho)$, while the above bound excludes any fixed point outside $\mathcal{B}(\rho)$,
which implies that $\u^*$ is the unique fixed point.
Therefore, $\|F(\u_k)\|\to 0$ implies that $\u_k$ eventually lies in the neighborhood $D$. Since $G$ is locally contractive on $D$, we have $(1-c)\|\u_k-\u^*\|\leq \|F(\u_k)\|$, which yields $\u_k\to\u^*$. Together with Assumption~\ref{anderson-ass2}, we have $\sum\nolimits_{j=0}^{m_k}\alpha_j^k \u_{k-m_k+j}\to \u^*$ and $\u_{\rm AA}^{k+1}\to \u^*$. Consequently, for all sufficiently large $k$, all the points involved in the local estimates of Theorems~\ref{second}, \ref{third} and \ref{fourth} lie in $D$.

Finally, we establish the linear convergence of GlobAA($m$) when $G$ is contractive. It remains only to verify that the sequence $\{\gamma_k\}$ defined in \eqref{eq_new_gamma} satisfies Assumption \ref{ass3}.
Due to $\frac{1}{k+2}\to 0$ and $W_{k}\to 0$, we have 
    \[
    \gamma_k =1- \max\left\{\frac{1}{k+2}, \frac{W_k}{2W_0}\right\}\to 1.
    \]
    Since $1-\gamma_k \geq \frac{1}{k+2}$,
    we have 
    \[
    \sum_{k=0}^{\infty} (1-\gamma_k)\geq \sum_{k=0}^{\infty} \frac{1}{k+2}=\infty.
    \]
Thus, $\{\gamma_k\}$ defined in \eqref{eq_new_gamma} satisfies Assumption \ref{ass3}, which implies that Theorems~\ref{second}, \ref{third} and \ref{fourth} follow. 
\end{proof}

Although the class-agnostic update strategy applies to both contractive and nonexpansive mappings, the separate assumptions remain useful when the mapping class is known. In the contractive case, the contraction itself drives the residual toward zero, so Assumption~\ref{ass3} allows much greater freedom in choosing $\gamma_k$: the sequence ${\gamma_k}$ can be prescribed independently of the residuals and approach one, thereby making the acceptance criterion progressively less restrictive and enabling GlobAA($m$) to eventually reduce to pure AA.
For a merely nonexpansive mapping, no strict contraction is available to ensure residual decay. Assumption~\ref{anderson-global-gamma} therefore requires sufficient reduction from accepted Anderson steps whenever $W_k$ stays away from zero. This permits both constant and residual-dependent choices of $\gamma_k$, while providing the decrease needed for global residual convergence.
Thus, the two assumptions play distinct roles and are generally not interchangeable. A sequence satisfying Assumption~\ref{ass3} need not satisfy Assumption~\ref{anderson-global-gamma}, whereas a constant choice $\gamma_k=\xi<1$ does not satisfy Assumption~\ref{ass3} and may prevent the eventual reduction to pure AA.
   
The class-agnostic strategy is particularly useful when the mapping class is unknown a priori or when the mapping exhibits different behaviors globally and locally. 
For example, consider
\[
G(\u)=\u-(1-a)\tanh(\u)\quad \mbox{with $a\in(0,1)$}.\]
Its Jacobian satisfies
\[
G'(\u)={\rm diag}\big(a+(1-a)\tanh^2u_1,\ldots,a+(1-a)\tanh^2u_n\big),\]
so $\|G'(\u)\|\leq1$ and $\sup_{\u}\|G'(\u)\|=1$. Hence, $G$ is globally nonexpansive but not globally contractive. Moreover, $\u^*={\bf 0}$ is its unique fixed point and $\|G'(\u^*)\|=a<1$, so $G$ is locally contractive around $\u^*$.
This example illustrates the situation above: without any change of the parameter rule, the class-agnostic strategy accommodates the global nonexpansive regime while exploiting the local contractive behavior near the fixed point.
When more information on $G$ is available, the separate assumptions may allow more tailored parameter choices; otherwise, the class-agnostic strategy provides a single rule applicable across different regimes.

\section{Numerical experiments}\label{sec3}
In this section, we present four numerical experiments to demonstrate the effectiveness of the proposed algorithm GlobAA($m$), and evaluate its superiority against several existing algorithms as follows.
\begin{itemize}
    \item Picard: Picard iteration with $\u_{k+1} = G(\u_k)$.
    \item KM: Krasnosel'ski\oldbreve{\i}--Mann iteration with $\u_{k+1} = (1-\lambda) \u_k + \lambda{G}(\u_k)$, and we set $\lambda=0.5$ by default. 
    \item pureAA($m$): Pure Anderson acceleration with a fixed memory parameter $m$ and without any globalization strategy. 
    \item resAA($m$): Anderson acceleration with residual-based globalization in  \cite{fu2020anderson}. We mainly consider the global framework from \cite[Algorithm 3]{fu2020anderson}, which employs a residual-based acceptance mechanism and uses the default parameter settings recommended in \cite[Section 7]{fu2020anderson}.
    \item FAA($m$): Filtered AA in \cite{pollock2023filtering}, in which history columns are selectively removed to control the condition number of the least-squares problem in \eqref{eq_coff}. 
    Here, $c_s$ specifies the minimum allowable sine of the angle between each column and the subspace spanned by the more recent columns.
\end{itemize}

All our experiments are performed in Python on a MacBook Pro (M3 Max, 128 GB of RAM).  We stop the algorithms when
\begin{equation*}
% \label{stop}
\frac{\|F(\u_k)\|}{\|F(\u_0)\|}\leq {\rm tol} \quad {\rm or } \quad k\geq k_{\max},
\end{equation*}
for a given tolerance ${\rm tol}$ and maximum number of iterations $k_{\max}$. Define $\y_k = F_{k+1}-F_k$, $\d_k = \u_{k+1}-\u_k$, $Y_k = [\y_{k-m_k}, \dots, \y_{k-1}]$, and $D_k = [\d_{k-m_k}, \dots, \d_{k-1}]$. Throughout the experiments, we transform \eqref{eq_coff} into the unconstrained optimization problem as follows:
\begin{equation}\label{uncon_opt}
\min \limits\nolimits_{{\r} \in \mathbb{R}^{m_k}} \left\|{F}_k-Y_k {\r}\right\|
\end{equation}
and set
\[
\u_{k+1}={G}_k-\sum\nolimits_{j=0}^{m_k-1}r_j^k({G}_{k-m_k+j+1}-{G}_{k-m_k+j}),
\]
where ${\r}^k = (r_0^k, \dots, r_{m_k-1}^k)^\top$ is a solution of \eqref{uncon_opt}. 

We solve the unregularized least-squares subproblem~\eqref{uncon_opt} using the \textsf{LSQR} method \cite{paige1982lsqr}. In some cases, we also test a regularized variant, in which the $\ell_2$-regularization term considered in~\cite{fu2020anderson} is added to~\eqref{uncon_opt}, yielding the problem
\begin{equation}\label{eq3.2}
\min \nolimits_{{\r} \in \mathbb{R}^{m_k}} \left\|{F}_k-Y_k {\r}\right\|^2 + \eta(\|D_k\|_F^2+\|Y_k\|_F^2)\|{\r}\|^2,
\end{equation}
with $\eta > 0$. 
Intuitively, if the algorithm converges, the coefficient of the regularization term vanishes as $k\rightarrow \infty$.

In subsection~\ref{sec3.1}, we present a simple example demonstrating that pureAA($m$) may fail to converge even for a contractive mapping that is differentiable at its fixed point. In contrast, our GlobAA($m$) method, based on the proposed globalization framework, enjoys global convergence and outperforms the approach proposed in~\cite{fu2020anderson}. In subsection~\ref{sec3.2}, we demonstrate the robustness of the proposed algorithm on a real-world dataset and an ill-conditioned problem, aiming to assess its performance in practical and numerically challenging scenarios. In subsections~\ref{sec3.3} and~\ref{sec3.4}, we demonstrate its efficiency for implicit fixed point problems using two PDE-based numerical examples that require little problem-specific prior information. 

We use Iter. to denote the number of iterations performed by the algorithms. In the following tables of this section, the smallest number of iterations and CPU time for different cases are shown in bold. AA(\%) denotes the percentage of accepted Anderson steps among all executed iterations. 
${\dagger}$ indicates that the stopping tolerance was not reached within $k_{\max}$ iterations.
In the figures of subsections~\ref{sec3.2}-\ref{sec3.4}, 
the iteration at which GlobAA($m$) performs its final KM step is marked by a red $\textcolor[rgb]{1.00,0.00,0.00}{\star}$ and a red vertical dashed line.

\subsection{A non-convergence example of pure Anderson acceleration}\label{sec3.1}
Consider the following optimization problem:
\begin{equation}\label{counter_ex}
    \min_{\u\in\mathbb{R}^n} f(\u) = n \phi(s(\u)) + \frac{1}{2}\|P\u\|^2,
\end{equation}
where $s(\u) = \frac{1}{n}\bf{1}^\top \u$ and $P = I - \frac{1}{n}\bf{1}\bf{1}^\top$ with ${\bf{1}} = (1,1,\dots,1)^\top \in \mathbb{R}^n$, and $\phi: \mathbb{R} \rightarrow \mathbb{R}$ is a one-dimensional counterexample constructed in \cite{mai2020anderson}, defined as follows: 
\begin{equation*}
    \phi(t)=
    \begin{cases}
        \frac{t^2}{20}-24.9t-12.45, & t<-1,\\
        12.5 t^2, & -1\leq t<1,\\
        \frac{t^2}{20}+24.9t-12.45, & t\geq1.
    \end{cases}
\end{equation*}
The gradient descent method with a fixed step size for problem \eqref{counter_ex} can be equivalently viewed as the Picard iteration for the following fixed point problem:
\begin{equation}\label{eq3.5}
    \u = G(\u) := \u - \tau \nabla f(\u),
\end{equation}
where $\tau > 0$ is the step size.
By $\nabla f(\u) = \phi'(s(\u)){\bf{1}} + P\u$, it is easy to verify that $f$ in \eqref{counter_ex} is strongly convex with $\rho = 1/10$ and $L$-smooth with $L=25$. Then, for $\tau \in (0, \frac{2}{\rho+L})$, $G$ is a contractive mapping with $c = \sqrt{1-\frac{2\rho\tau L}{\rho+L}}$. Following \cite{mai2020anderson}, we set $\tau = 1/L$, and $G$ has the unique fixed point at $\u^*=0$.

For solving \eqref{eq3.5}, the following proposition characterizes how the convergence behavior of pureAA(1) varies with the choice of the initial iterate.
\begin{proposition}\label{prop:counter_ex}
    Denote $\mathcal{D}:={\rm span}\{{\bf{1}}\}=\{t{\bf{1}} : t \in \mathbb{R}\}$. For any initial iterate \(\u_0=t_0 \bf{1}\in\mathcal D\), the iterates generated by pureAA(1) for solving problem \eqref{eq3.5} remain in \(\mathcal D\).
    In particular, if \(t_0\in[2.01,246.98]\), then the sequence \(\{\u_k\}\) does not converge to the unique fixed point, and it
    approaches the following period-four orbit:
    \begin{equation}\label{eq_pureAA_res}
    \begin{aligned}
    &\u_{4k+3} \rightarrow -249(\sqrt{5}-2){\bf{1}}, \quad \u_{4k+4} \rightarrow +249{\bf{1}},\\
    &\u_{4k+5} \rightarrow +249(\sqrt{5}-2){\bf{1}}, \quad \u_{4k+6} \rightarrow -249{\bf{1}}.
    \end{aligned}
    \end{equation}
\end{proposition}
\begin{proof}
    We first prove by induction that \(\mathcal{D}\) is invariant under the iteration of pureAA($1$).
    The conclusion clearly holds for $k=0$. Suppose that $\u_i\in \mathcal{D}$ for all $i=0,1,\dots,k$. For simplicity, we denote $\u_i = t_i {\bf{1}}$ for some $t_i\in \mathbb{R}$. Then, we have $s(\u_i) = t_i$ and $P\u_i = \mathbf{0}$ for all $i=0,1,\dots,k$.
    Together with $\nabla f(\u_i) = \phi'(t_i){\bf{1}}$, we have $\nabla f(\u_i)\in \mathcal{D}$ for all $i=0,1,\dots,k$. Thus, we have 
    \begin{equation}\label{sec31_eq1}
        G(\u_i) = G(t_i {\bf{1}}) = t_i {\bf{1}} - \tau \nabla f(t_i {\bf{1}})=(t_i - \tau \phi'(t_i)){\bf{1}},
    \end{equation}
    which means $G(\u_i)\in \mathcal{D}$ for all $i=0,1,\dots,k$. From \eqref{eq_pureAA}, we have $\u_{k+1}\in \mathcal{D}$. 
    
    By \eqref{sec31_eq1} and $\u_i \in \mathcal{D}$, $\mathcal{D}$ is invariant under $F$. Thus, the combination coefficients in \eqref{uncon_opt} coincide with those of the one-dimensional problem in \cite{mai2020anderson}, which implies that the iteration of \eqref{eq_pureAA} for \eqref{eq3.5} coincides with the one-dimensional Anderson iteration in \cite{mai2020anderson}.
    By \cite[Proposition~1]{mai2020anderson}, for any initial iterate $\u_0=t_0 \bf{1}$ with $t_0\in[2.01, 246.98]$, the sequence generated by pureAA(1) fails to converge to the fixed point and approaches a period-four orbit as in \eqref{eq_pureAA_res}.
\end{proof}
    
\textbf{Implementation details.} We set $n=200$, $k_{\max} = 500$, ${\rm tol} = 10^{-8}$, and compare the proposed GlobAA($m$) with Picard iteration, pureAA($m$) and resAA($m$). 
We consider the memory sizes $m\in\{1,3,5\}$ and three initial iterates
\[
\u_0\sim\mathcal{N}(0,I_n), \quad \u_0 = 2.1 \,{\bf{1}}, \quad \u_0 = 300 \,{\bf{1}}.
\]
We first investigate the influence of $\gamma_k$ on the performance of GlobAA($m$) by testing several choices of $\gamma_k$ that satisfy Assumption~\ref{ass3}.
For resAA($m$), we use the parameter settings recommended in \cite[Section 7]{fu2020anderson}.

We then choose $\gamma_k = \frac{k+300}{k+301}$, $\lambda=1$ and $\eta = 0$ in GlobAA($m$) for the subsequent comparisons, and focus on the two initial iterates in the invariant subspace $\mathcal{D}={\rm span}\{\bf{1}\}$. 
We compare the four algorithms for these two representative initial iterates, respectively, and investigate the influence of the memory size $ m$. We also indicate the step type at each iteration of GlobAA($m$) along the $x$-axis, where $\textcolor[rgb]{1.00,0.00,0.00}{\star}$ denotes an iteration in which an Anderson step is performed, while $\textcolor[rgb]{0.11,0.61,0.88}{ \circ}$ denotes an iteration using a Picard/KM step.

\textbf{Experimental results.} Based on the numerical results reported in Table~\ref{tab3.1:gamma_sweep} and Figs. \ref{fig3_1}--\ref{fig3_3}, we make the following observations.
\begin{itemize}
\item For the general initial iterate $\u_0\sim\mathcal{N}(0,I_n)$, although all the algorithms meet the stopping criterion within $k_{\max}$ iterations, the Anderson-based methods, including pureAA($m$), resAA($m$) and GlobAA($m$), generally require substantially fewer iterations than Picard, illustrating the acceleration effect of Anderson methods.
\item GlobAA($m$) converges for all the tested settings, including different choices of the initial iterate, the memory size $m$, and the parameter $\gamma_k$, while the choice of $\gamma_k$ can noticeably affect the number of iterations and the percentage of accepted Anderson steps. In contrast, the other two Anderson acceleration methods, pureAA($m$) and resAA($m$), fail to converge for some cases. 
\item Figs.~\ref{fig3_1}--\ref{fig3_2} show the convergence behavior of the tested algorithms in two different settings, where $\gamma_k = \frac{k+300}{k+301}$ is used for GlobAA($m$).
As shown in Fig.~\ref{fig3_1}, the behavior of pureAA(1) with $\u_0=2.1 \,{\bf{1}}$ is fully consistent with the statements in Proposition~\ref{prop:counter_ex}: instead of converging to the unique fixed point, the iterates approach the period-four orbit in \eqref{eq_pureAA_res}. Interestingly, a similar period-four behavior is also observed for pureAA(1) with $\u_0=300\,{\bf{1}}$, as illustrated in Fig.~\ref{fig3_2}, although this initial iterate lies outside the interval covered by Proposition~\ref{prop:counter_ex}. In both cases, resAA(1) exhibits essentially the same behavior as pureAA(1), with its Anderson steps continuously accepted, and therefore fails to eliminate the periodic behavior within the maximum number of iterations. In contrast, GlobAA(1) converges to the fixed point from both initial iterates. As illustrated by the step histories, the globalization mechanism rejects unsuitable Anderson steps and invokes fallback steps when necessary, thereby safeguarding the iteration against persistent non-convergent behavior.
\item Fig.~\ref{fig3_3} compares the algorithms with $\u_0=2.1 \,{\bf{1}}$ for different memory sizes $m$. We see that pureAA(3) and resAA(3) remain non-convergent over the tested iterations, whereas both methods converge when $m=5$, illustrating their sensitivity to the memory size. In contrast, GlobAA($m$) converges rapidly for both $m=3$ and $m=5$. Together with the results above, these experiments demonstrate that the proposed globalization strategy improves the robustness of Anderson acceleration.
\end{itemize}

\begin{table}[htbp]
\centering
\caption{Numerical results of different algorithms for solving \eqref{eq3.5}}
\label{tab3.1:gamma_sweep}
\tiny
\setlength{\tabcolsep}{3pt}
\renewcommand{\arraystretch}{1.12}

% \resizebox{\textwidth}{!}{
\begin{tabular}{llc cc cc cc}
\toprule
\multirow{2}{*}{Case}
& \multirow{2}{*}{Algorithm}
& \multirow{2}{*}{$\gamma_k$}
& \multicolumn{2}{c}{$m=1$}
& \multicolumn{2}{c}{$m=3$}
& \multicolumn{2}{c}{$m=5$} \\
\cmidrule(lr){4-5}
\cmidrule(lr){6-7}
\cmidrule(lr){8-9}
&
&
& Iter. & AA(\%)
& Iter. & AA(\%)
& Iter. & AA(\%) \\
\midrule

% ============================================================
% u_0 = randn(seed=42)
% ============================================================
\multirow{9}{*}{\shortstack{$\u_0\sim\mathcal{N}(0,I_n)$}}
& \rule[-0.8ex]{0pt}{3.2ex}Picard 
& --
& 439 & 0.00
& 439 & 0.00
& 439 & 0.00 \\

& \rule[-0.8ex]{0pt}{3.2ex}pureAA($m$)
& --
& \textbf{3} & 100.00
& \textbf{3} & 100.00
& \textbf{3} & 100.00 \\

& \rule[-0.8ex]{0pt}{3.2ex}resAA($m$)
& --
& \textbf{3} & 100.00
& 4 & 100.00
& 4 & 100.00 \\
\noalign{\vskip 0.6ex}
\cdashline{2-9}[2.5pt/1.5pt]
\noalign{\vskip 0.6ex}

& \rule[-0.8ex]{0pt}{3.2ex}GlobAA($m$)
& $1-\tfrac{1}{k+2}$
& 26 & 7.69
& 26 & 7.69
& 26 & 7.69 \\

& \rule[-0.8ex]{0pt}{3.2ex}GlobAA($m$)
& $1-\tfrac{1}{k+301}$
& \textbf{3} & 100.00
& \textbf{3} & 100.00
& \textbf{3} & 100.00 \\

& \rule[-0.8ex]{0pt}{3.2ex}GlobAA($m$)
& $1-\tfrac{1}{5k+2}$
& 7 & 28.57
& 7 & 28.57
& 7 & 28.57 \\

& \rule[-0.8ex]{0pt}{3.2ex}GlobAA($m$)
& $1-\tfrac{1}{100k+2}$
& \textbf{3} & 66.67
& \textbf{3} & 66.67
& \textbf{3} & 66.67 \\

& \rule[-0.8ex]{0pt}{3.2ex}GlobAA($m$)
& $1-\tfrac{1}{\sqrt{k+2}}$
& 439 & 0.00
& 439 & 0.00
& 439 & 0.00 \\

& \rule[-0.8ex]{0pt}{3.2ex}GlobAA($m$)
& $1-\tfrac{1}{10\log(k+2)}$
& \textbf{3} & 100.00
& \textbf{3} & 100.00
& \textbf{3} & 100.00 \\

\midrule

% ============================================================
% u_0 = 2.1 * 1
% ============================================================
\multirow{9}{*}{$\u_0=2.1\,\mathbf{1}$}
& \rule[-0.8ex]{0pt}{3.2ex}Picard
& --
& \textbf{3} & 0.00
& \textbf{3} & 0.00
& \textbf{3} & 0.00 \\

& \rule[-0.8ex]{0pt}{3.2ex}pureAA($m$)
& --
& $500^{\dagger}$ & 100.00
& $500^{\dagger}$ & 100.00
& 32 & 100.00 \\

& \rule[-0.8ex]{0pt}{3.2ex}resAA($m$)
& --
& $500^{\dagger}$ & 100.00
& $500^{\dagger}$ & 100.00
& 31 & 100.00 \\
\noalign{\vskip 0.6ex}
\cdashline{2-9}[2.5pt/1.5pt]
\noalign{\vskip 0.6ex}

& \rule[-0.8ex]{0pt}{3.2ex}GlobAA($m$)
& $1-\tfrac{1}{k+2}$
& 4 & 75.00
& 6 & 83.33
& 8 & 87.50 \\

& \rule[-0.8ex]{0pt}{3.2ex}GlobAA($m$)
& $1-\tfrac{1}{k+301}$
& \textbf{3} & 66.67
& \textbf{3} & 66.67
& \textbf{3} & 66.67 \\

& \rule[-0.8ex]{0pt}{3.2ex}GlobAA($m$)
& $1-\tfrac{1}{5k+2}$
& 4 & 75.00
& 6 & 83.33
& 8 & 87.50 \\

& \rule[-0.8ex]{0pt}{3.2ex}GlobAA($m$)
& $1-\tfrac{1}{100k+2}$
& 4 & 75.00
& 6 & 83.33
& 8 & 87.50 \\

& \rule[-0.8ex]{0pt}{3.2ex}GlobAA($m$)
& $1-\tfrac{1}{\sqrt{k+2}}$
& 4 & 75.00
& 6 & 83.33
& 8 & 87.50 \\

& \rule[-0.8ex]{0pt}{3.2ex}GlobAA($m$)
& $1-\tfrac{1}{10\log(k+2)}$
& 4 & 75.00
& 6 & 83.33
& 8 & 87.50 \\

\midrule

% ============================================================
% u_0 = 300 * 1
% ============================================================
\multirow{9}{*}{$\u_0=300\,\mathbf{1}$}
& \rule[-0.8ex]{0pt}{3.2ex}Picard 
& --
& 198 & 0.00
& 198 & 0.00
& 198 & 0.00 \\

& \rule[-0.8ex]{0pt}{3.2ex}pureAA($m$)
& --
& $500^{\dagger}$ & 100.00
& 24 & 100.00
& \textbf{30} & 100.00 \\

& \rule[-0.8ex]{0pt}{3.2ex}resAA($m$)
& --
& $500^{\dagger}$ & 100.00
& 21 & 100.00
& \textbf{30} & 100.00 \\
\noalign{\vskip 0.6ex}
\cdashline{2-9}[2.5pt/1.5pt]
\noalign{\vskip 0.6ex}

& \rule[-0.8ex]{0pt}{3.2ex}GlobAA($m$)
& $1-\tfrac{1}{k+2}$
& 199 & 1.51
& 201 & 2.49
& 203 & 3.45 \\

& \rule[-0.8ex]{0pt}{3.2ex}GlobAA($m$)
& $1-\tfrac{1}{k+301}$
& \textbf{32} & 56.25
& 24 & 100.00
& \textbf{30} & 100.00 \\

& \rule[-0.8ex]{0pt}{3.2ex}GlobAA($m$)
& $1-\tfrac{1}{5k+2}$
& 199 & 38.19
& 201 & 38.81
& 203 & 39.41 \\

& \rule[-0.8ex]{0pt}{3.2ex}GlobAA($m$)
& $1-\tfrac{1}{100k+2}$
& 36 & 52.78
& \textbf{20} & 80.00
& 33 & 90.91 \\

& \rule[-0.8ex]{0pt}{3.2ex}GlobAA($m$)
& $1-\tfrac{1}{\sqrt{k+2}}$
& 199 & 1.51
& 201 & 2.49
& 203 & 3.45 \\

& \rule[-0.8ex]{0pt}{3.2ex}GlobAA($m$)
& $1-\tfrac{1}{10\log(k+2)}$
& 199 & 1.51
& 201 & 2.49
& 203 & 3.45 \\

\bottomrule
\end{tabular}
% }
\end{table}

\begin{figure}[htbp]
    \centering
    \includegraphics[width=0.99\textwidth]{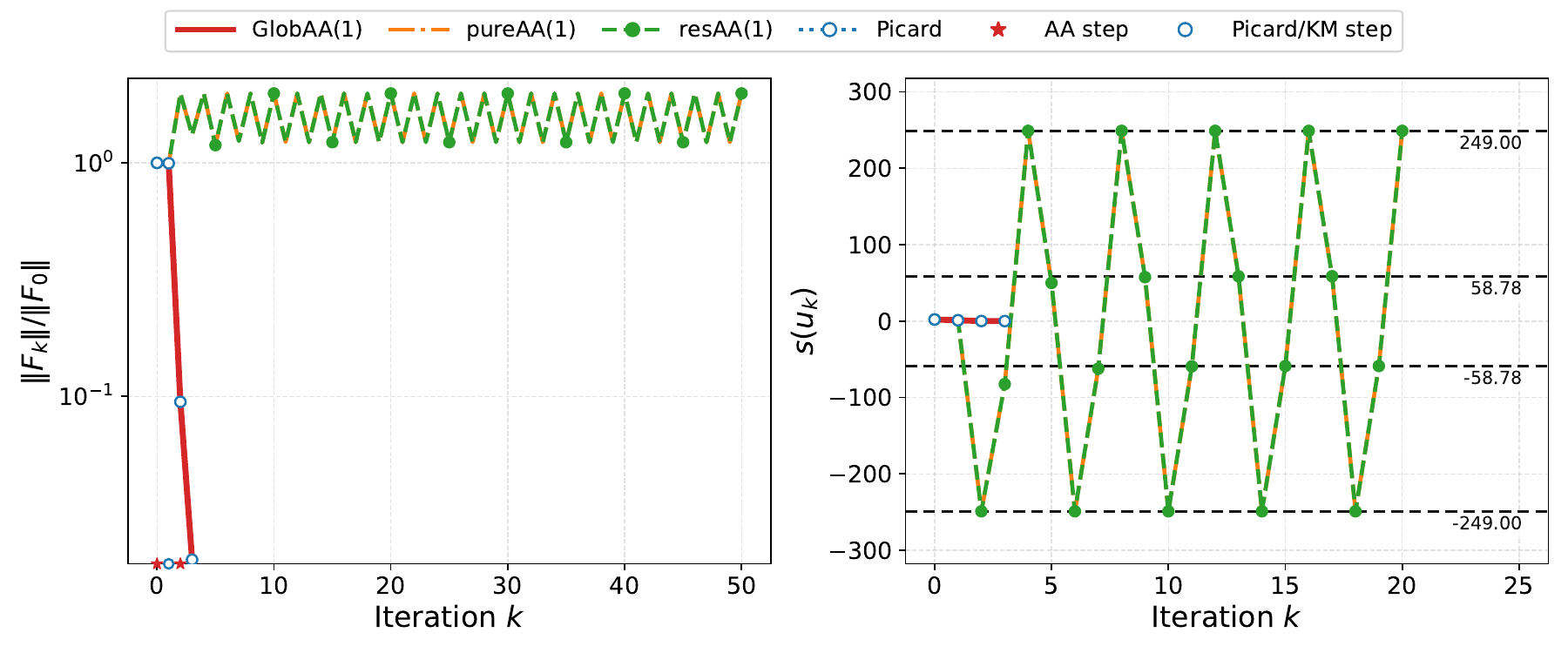}
    \vspace{-8pt}
    \caption{Convergence results of different algorithms for $\u_0 = 2.1 \,{\bf{1}}$.}\label{fig3_1}
\end{figure}

\begin{figure}[htbp]

    \centering
    \includegraphics[width=0.99\textwidth]{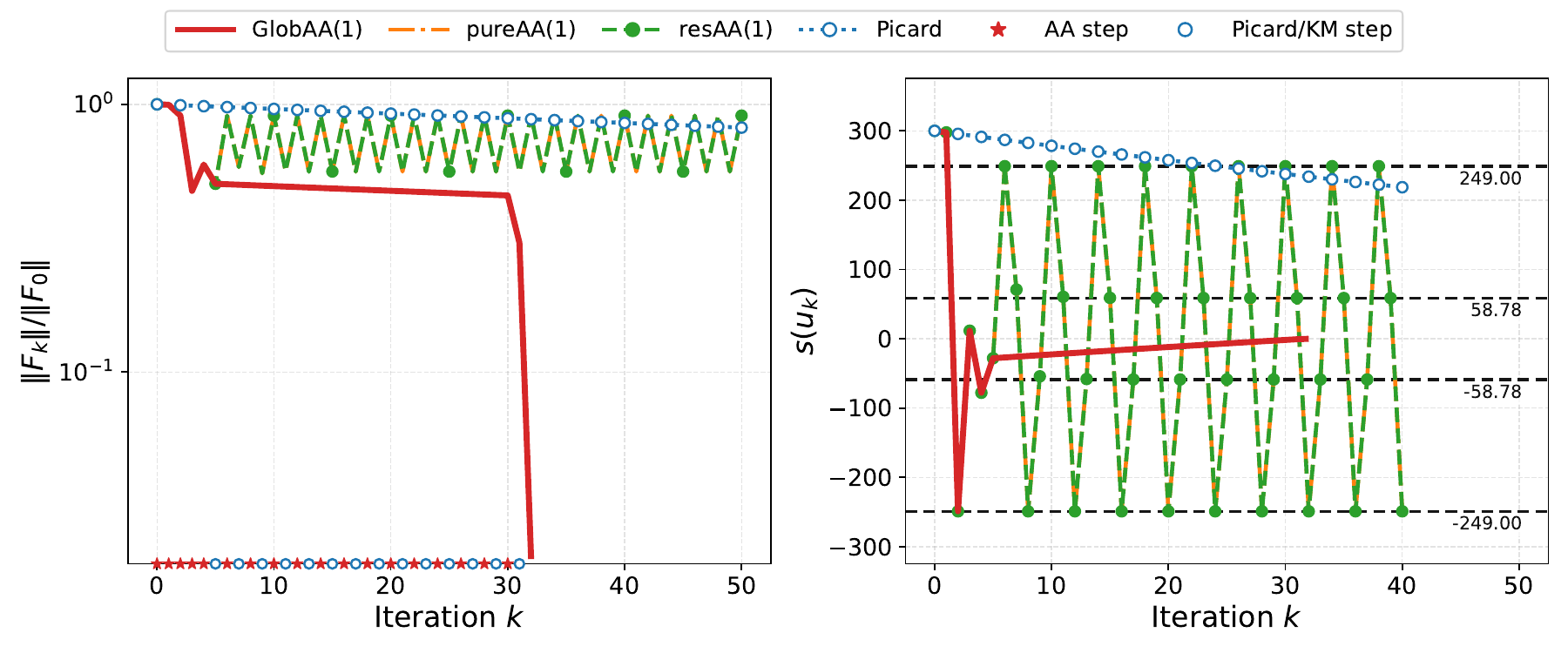}
    \caption{Convergence results of different algorithms for $\u_0 = 300 \,{\bf{1}}$.}\label{fig3_2}
\end{figure}

\begin{figure}[htbp]
    \centering
    \includegraphics[width=0.99\textwidth]{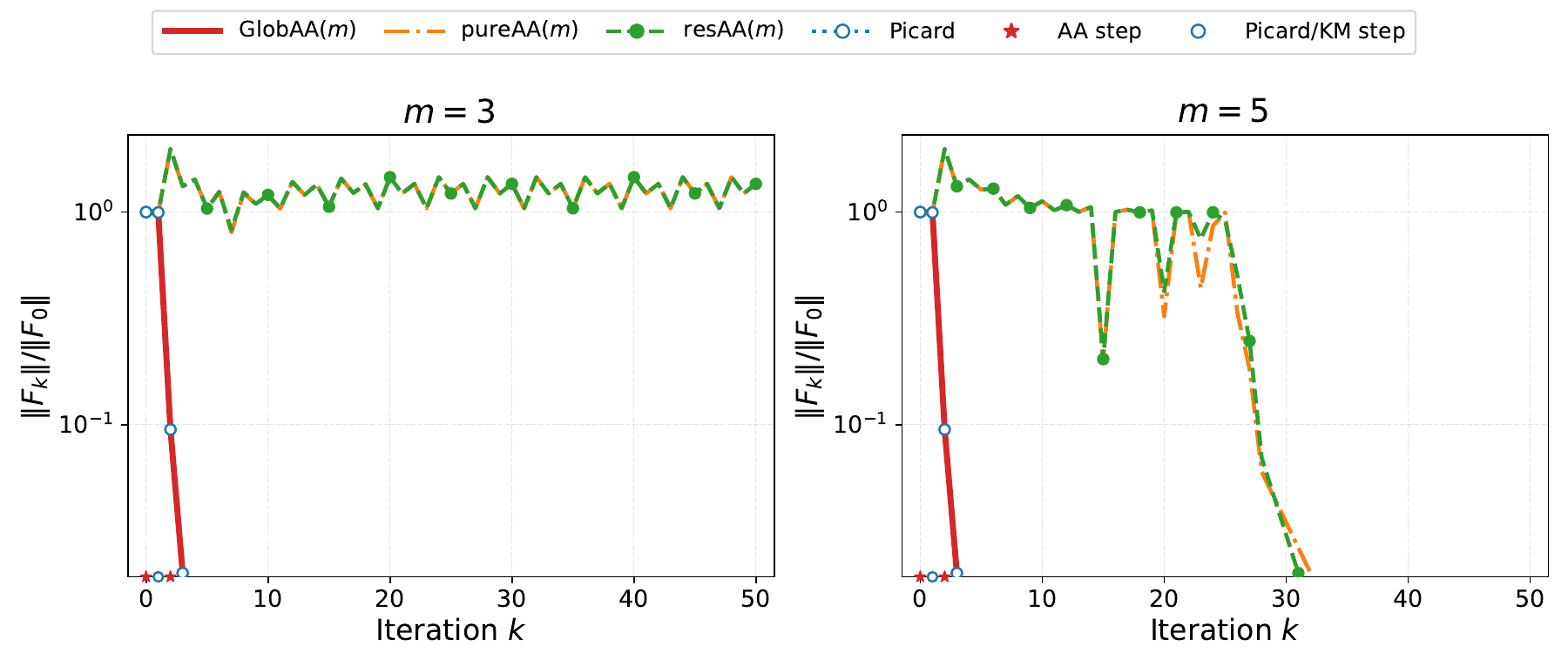}
    \caption{Convergence results of different algorithms for different values of $m$.}\label{fig3_3}
\end{figure}

\subsection{Elastic net regression}\label{sec3.2} 
In this section, we consider the following elastic net regression (ENR) problem \cite{zou2005regularization}:
\begin{equation}\label{ENR}
    \min_{\u\in\mathbb{R}^n} \frac{1}{2}\|A\u-\b\|^2 + \mu\left(\frac{1-{\beta}}{2}\|\u\|^2 + {\beta}\|\u\|_1\right),
\end{equation}
where $A\in\mathbb{R}^{M\times n}$ and $\b\in\mathbb{R}^M$. Following \cite{o2016conic, zhang2020globally}, we define $\mu_{\rm max}=\|A^\top\b\|_{\infty}$, which is used as a reference scale for selecting the regularization parameter $\mu$.
ENR is a regularized regression method that combines the strengths of Lasso ($\ell_1$) and Ridge ($\ell_2$) regression to handle multicollinearity and perform variable selection. 
Applying the iterative shrinkage-thresholding algorithm (ISTA) \cite{beck2009fast} to solve the ENR problem in \eqref{ENR} yields the following fixed point iteration:
\begin{equation*}
    \u_{k+1} = G(\u_k) := \mathcal{S}_{\tau\mu\beta}\left(\u_k - \tau [A^\top(A\u_k - \b)+\mu(1-\beta)\u_k]\right),
\end{equation*}
where $\mathcal{S}_{\tau\mu\beta}$ is the soft-thresholding operator defined as $\mathcal{S}_{\tau\mu\beta}(\u)_i = \text{sign}(\u_i)(|\u_i| - \tau\mu\beta)_+$ for $i=1,\ldots,n$.

\textbf{Implementation details.} In this subsection, we generate a random vector $\boldsymbol{X}\in\mathbb{R}^n$ with sparsity 0.1, and set $\b$ by $\b = A \boldsymbol{X} + 0.1 \boldsymbol{\omega}$, where the matrix $A$ is given in each set of experiments, and $\boldsymbol{\omega}\in\mathbb{R}^M$ is a standard Gaussian random vector, i.e., $\boldsymbol{\omega}\sim\mathcal{N}(0,I_M)$.
Throughout this subsection, we test GlobAA($m$) both with and without regularization, where the regularized GlobAA($m$) uses the same regularization term as resAA(\(m\)) in the Anderson coefficient subproblem, and we use $\u_0 = {0}$ as the initial iterate for all compared algorithms. We consider the following two numerical experiments for the ENR problem \eqref{ENR}.
\begin{itemize}
    \item Exp. 1 ($G$ is a contractive mapping): Set ${\rm tol} = 10^{-10}$, $k_{\max} = 2000$, $\beta = 1/2$, $\mu = 10^{-5}\mu_{\rm max}$, and $\tau = 1.8/L$ with {\bf $L = \lambda_{\max}(A^\top A) + \mu/2$}. With this definition of $L$, the mapping $G$ is contractive for any $\tau\in(0,2/L)$. We choose $\gamma_k = \frac{k+1}{k+2}$ and $\lambda=1$ for GlobAA($m$), and use $\eta=10^{-8}$ in \eqref{eq3.2} for both GlobAA($m$) with regularization and resAA($m$). Matrix $A\in \mathbb{R}^{2000\times 500}$ is taken from the real-world Madelon dataset\footnote[1]{The Madelon dataset is available from \url{https://archive.ics.uci.edu/dataset/171/madelon}}.
    \item Exp. 2 ($G$ is a nonexpansive mapping): Set ${\rm tol} = 10^{-10}$, $k_{\max} = 25000$, $\beta = 1$, $\mu = 10^{-3}\mu_{\rm max}$, and $\tau = 2/L$ with $L = \lambda_{\max}(A^\top A)$. We choose $\gamma_k = 1-10^{-5}<1$ and $\lambda=0.5$ for GlobAA($m$), and use $\eta=10^{-10}$ in \eqref{eq3.2} for both GlobAA($m$) with regularization and resAA($m$). We generate $A\in\mathbb{R}^{1000\times2000}$ with random orthonormal singular vectors and singular values logarithmically spaced from $1$ to $10^{-6}$, yielding a condition number of $10^6$.
\end{itemize}

\textbf{Experimental results.} Table~\ref{tab3.4:comparison} together with Figs.~\ref{fig3.2-1}--\ref{fig3.2-2} shows the convergence results for the two experiments and indicates that GlobAA($m$) is more robust than the comparison methods.
For both contractive and nonexpansive cases, GlobAA($m$) with regularization consistently reaches the prescribed tolerance and generally 
requires less CPU time than the compared methods, with its advantage becoming increasingly pronounced as $m$ grows, while the competing methods Picard, KM, pureAA($m$), and resAA($m$) may fail to converge for large $m$.
Moreover, GlobAA($m$) with regularization requires substantially less CPU time in most cases, while the corresponding iteration counts remain generally comparable.
In particular, for the nonexpansive case, Picard fails to meet the stopping criterion for all tested values of $m$, while pureAA($m$) fails for most of the tested values of $m$.

% k+1/k+2
% \vspace{-0.03in}
\begin{table}[htbp]
\centering
\caption{Numerical results of different algorithms for the ENR problem in \eqref{ENR}}
\label{tab3.4:comparison}
\tiny
\setlength{\tabcolsep}{2pt}
\renewcommand{\arraystretch}{1.12}
\resizebox{\textwidth}{!}{
\begin{tabular}{llc ccc ccc ccc ccc ccc}
\toprule
\multirow{2}{*}{Experiment}
& \multirow{2}{*}{Algorithm}
& \multirow{2}{*}{$\eta$}
& \multicolumn{3}{c}{$m=3$}
& \multicolumn{3}{c}{$m=5$}
& \multicolumn{3}{c}{$m=10$}
& \multicolumn{3}{c}{$m=20$}
& \multicolumn{3}{c}{$m=30$} \\
\cmidrule(lr){4-6}
\cmidrule(lr){7-9}
\cmidrule(lr){10-12}
\cmidrule(lr){13-15}
\cmidrule(lr){16-18}
&
&
& Iter. & Time(s) & AA(\%)
& Iter. & Time(s) & AA(\%)
& Iter. & Time(s) & AA(\%)
& Iter. & Time(s) & AA(\%)
& Iter. & Time(s) & AA(\%) \\
\midrule

\multirow{6}{*}{\shortstack{Exp. 1\\Madelon\\contractive}}
& Picard
& $0$
& $2000^{\dagger}$ & 0.17 & 0.00
& $2000^{\dagger}$ & 0.17 & 0.00
& $2000^{\dagger}$ & 0.17 & 0.00
& $2000^{\dagger}$ & 0.17 & 0.00
& $2000^{\dagger}$ & 0.17 & 0.00 \\

& KM, $\lambda=0.5$
& $0$
& $2000^{\dagger}$ & 0.17 & 0.00
& $2000^{\dagger}$ & 0.17 & 0.00
& $2000^{\dagger}$ & 0.17 & 0.00
& $2000^{\dagger}$ & 0.17 & 0.00
& $2000^{\dagger}$ & 0.17 & 0.00 \\

& pureAA$(m)$
& $0$
& 629 & 0.24 & 100.00
& 645 & 0.33 & 100.00
& $2000^{\dagger}$ & 2.07 & 100.00
& $2000^{\dagger}$ & 3.78 & 100.00
& $2000^{\dagger}$ & 5.62 & 100.00 \\

& resAA$(m)$
& $10^{-8}$
& \textbf{620} & \textbf{0.12} & 100.00
& 761 & \textbf{0.14} & 100.00
& 1479 & 0.29 & 100.00
& $2000^{\dagger}$ & 0.45 & 100.00
& $2000^{\dagger}$ & 0.55 & 100.00 \\

& GlobAA$(m)$
& $0$
& 767 & 0.23 & 55.54
& \textbf{554} & 0.17 & 37.00
& 522 & 0.20 & 23.37
& 199 & 0.24 & 84.92
& \textbf{175} & 0.32 & 93.71 \\

& GlobAA$(m)$
& $10^{-8}$
& 797 & 0.15 & 76.91
& 995 & 0.19 & 99.40
& \textbf{221} & \textbf{0.05} & 97.74
& \textbf{173} & \textbf{0.04} & 97.11
& 223 & \textbf{0.06} & 85.65 \\

\midrule

\multirow{6}{*}{\shortstack{Exp. 2\\Ill-conditioned\\($\kappa(A)=10^6$)\\nonexpansive}}
& Picard
& $0$
& $25000^{\dagger}$ & 6.56 & 0.00
& $25000^{\dagger}$ & 6.56 & 0.00
& $25000^{\dagger}$ & 6.56 & 0.00
& $25000^{\dagger}$ & 6.56 & 0.00
& $25000^{\dagger}$ & 6.56 & 0.00 \\

& KM, $\lambda=0.5$
& $0$
& $25000^{\dagger}$ & 6.53 & 0.00
& $25000^{\dagger}$ & 6.53 & 0.00
& $25000^{\dagger}$ & 6.53 & 0.00
& $25000^{\dagger}$ & 6.53 & 0.00
& $25000^{\dagger}$ & 6.53 & 0.00 \\

& pureAA$(m)$
& $0$
& $25000^{\dagger}$ & 17.05 & 100.00
& $25000^{\dagger}$ & 21.20 & 100.00
& 23096 & 32.74 & 100.00
& $25000^{\dagger}$ & 80.08 & 100.00
& $25000^{\dagger}$ & 147.40 & 100.00 \\

& resAA$(m)$
& $10^{-10}$
& \textbf{11854} & 4.68 & 100.00
& \textbf{12206} & 5.10 & 100.00
& 15467 & 7.19 & 100.00
& $25000^{\dagger}$ & 17.60 & 100.00
& $25000^{\dagger}$ & 19.34 & 100.00 \\

& GlobAA$(m)$
& $0$
& $25000^{\dagger}$ & 12.61 & 90.29
& 15396 & 8.93 & 84.39
& 12785 & 9.50 & 80.82
& 11874 & 13.92 & 78.35
& \textbf{11096} & 17.90 & 78.43 \\

& GlobAA$(m)$
& $10^{-10}$
& 12573 & \textbf{4.00} & 80.18
& 12292 & \textbf{3.98} & 81.92
& \textbf{11905} & \textbf{4.17} & 79.69
& \textbf{11703} & \textbf{5.26} & 80.28
& 11893 & \textbf{5.64} & 80.30 \\

\bottomrule
\end{tabular}
}
\end{table}
% \vspace{-0.2in}
\begin{figure}[htbp]
    \centering
    \includegraphics[width=0.99\textwidth]{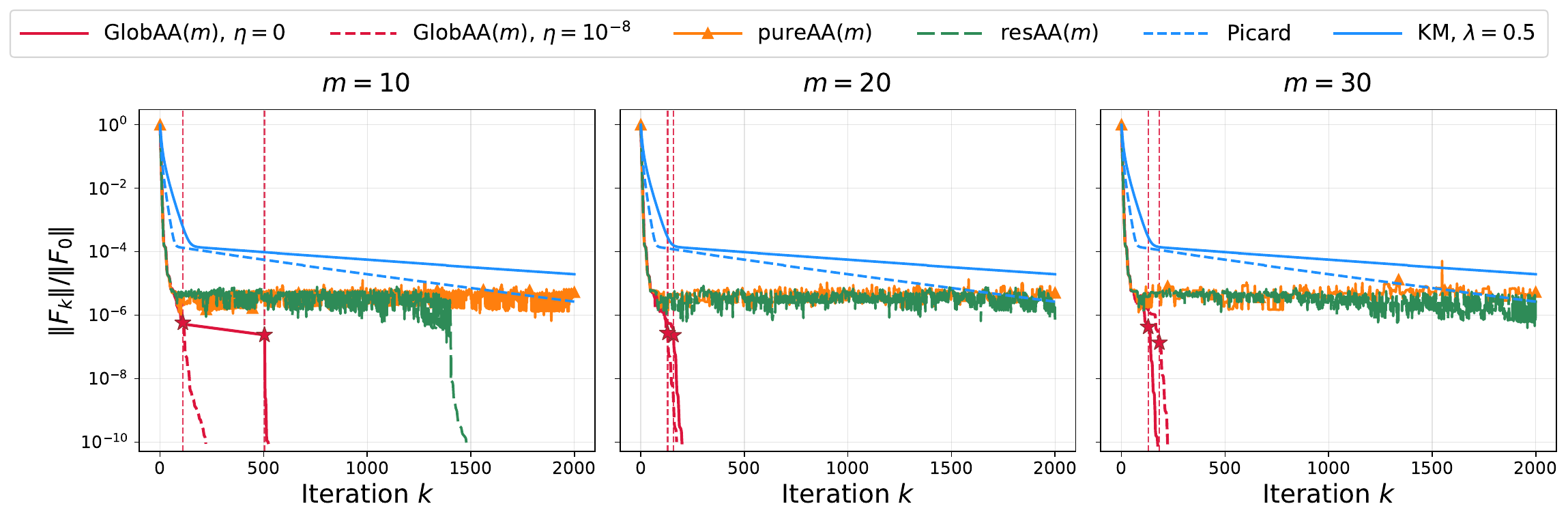}
    \caption{Convergence behavior of different algorithms for Exp. 1 (contractive).}\label{fig3.2-1}
\end{figure}
% \vspace{-0.2in}
\begin{figure}[htbp]
    \centering
    \includegraphics[width=0.99\textwidth]{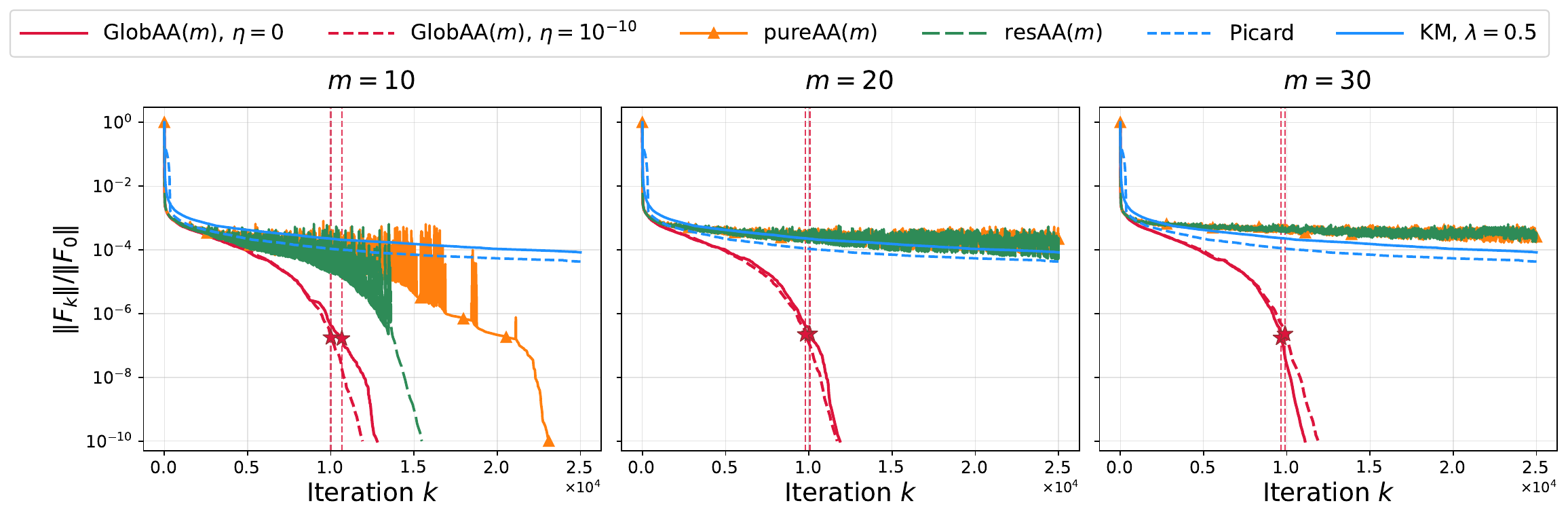}
    \caption{Convergence behavior of different algorithms for Exp. 2 (nonexpansive).}\label{fig3.2-2}
\end{figure}

\subsection{Nonlinear Helmholtz equation}\label{sec3.3}
In this subsection, we consider the nonlinear Helmholtz (NLH) equation, which arises in nonlinear optics and describes the propagation of continuous-wave laser beams through transparent dielectrics. The one-dimensional NLH system \cite{baruch2007high} can be expressed as: find $v: [0, 10] \rightarrow \mathbb{C}$, satisfying 
\begin{equation*}
    \begin{aligned}
    \frac{\mathrm{d}^2 v}{\mathrm{d} x^2} + k_0^2(1+\epsilon |v|^2)v &= 0, \quad 0<x<10,\\
    \frac{\mathrm{d} v}{\mathrm{d} x} + i k_0 v &= 2 i k_0, \quad x=0,\\
    \frac{\mathrm{d} v}{\mathrm{d} x} - i k_0 v &= 0, \quad x=10,
    \end{aligned}
\end{equation*}
where $v=v(x)$ denotes the unknown complex-valued scalar electric field, $k_0$ is the linear wave number in the surrounding medium, and $\epsilon>0$ is a material parameter determined by the linear refractive index and the Kerr coefficient. Although $\epsilon$ may generally depend on the spatial variable, it is assumed to be constant here.
Despite its one-dimensional formulation, this problem remains challenging for nonlinear solvers because of the cubic nonlinearity \cite{baruch2007high, baruch2009high}.

We discretize the following iteration using a second-order finite difference method: \begin{equation}\label{NLH}
    \begin{aligned}
    \frac{\mathrm{d}^2 v_{j+1}}{\mathrm{d} x^2} + k_0^2(1+\epsilon |v_{j}|^2)v_{j+1} &= 0, \quad 0<x<10,\\
    \frac{\mathrm{d} v_{j+1}}{\mathrm{d} x} + i k_0 v_{j+1} &= 2 i k_0, \quad x=0,\\
    \frac{\mathrm{d} v_{j+1}}{\mathrm{d} x} - i k_0 v_{j+1} &= 0, \quad x=10.
    \end{aligned}
\end{equation}
Let $\u_j$ denote the discrete approximation of $v_j$, and let $G$ be the solution operator of the resulting discretized linear system. Then the iteration can be written as the fixed point iteration $\u_{j+1} = G(\u_j)$. 
By separating the real and imaginary parts, this complex-valued discrete fixed point problem can be equivalently represented in a real Euclidean space, consistent with the framework considered in Section~\ref{sec2}.

\textbf{Implementation details.} We set $k_0 = 8$ and $\epsilon = 0.2$ in \eqref{NLH}, and we use $N=2001$ for the discretized system. Following \cite{baruch2007high}, the initial vector $\u_0$ is obtained by discretizing $v_0= e^{i k_0 x}$.
We set $k_{\max} = 500$, ${\rm tol} = 10^{-10}$, $\lambda=0.5$ and $\gamma_k = 1 - \max\left\{\frac{1}{k+2}, \frac{W_k}{2 W_0}\right\}$ as in \eqref{eq_new_gamma} for GlobAA($m$). For FAA($m$), we use $\bar{\kappa}=10^8$ and test $c_s=0.1$ and $c_s=0.2$. 
We also test GlobAA($m$) using \eqref{eq3.2} with regularization parameter $\eta=10^{-10}$. 

\textbf{Experimental results.}
Fig.~\ref{fig3.3} and Table~\ref{tab3.3:comparison} report the convergence results of different algorithms with different memory sizes. It can be observed that GlobAA($m$) generally requires fewer iterations than the competing methods across the tested memory sizes. Moreover, GlobAA($m$) not only demonstrates strong robustness with respect to $m$, but also tends to perform better as $m$ increases.
Specifically, the Picard iteration stagnates for this problem, while pureAA($m$) is sensitive to the choice of the memory size \(m\). Moreover, for GlobAA($m$), the regularized variant with \(\eta=10^{-10}\) exhibits almost the same convergence behavior as the unregularized version, while significantly reducing the CPU time. 
This indicates that the regularized implementation is computationally more efficient without noticeably changing the overall iteration process.
\begin{figure}[htbp]
    \centering
    \includegraphics[width=0.99\textwidth]{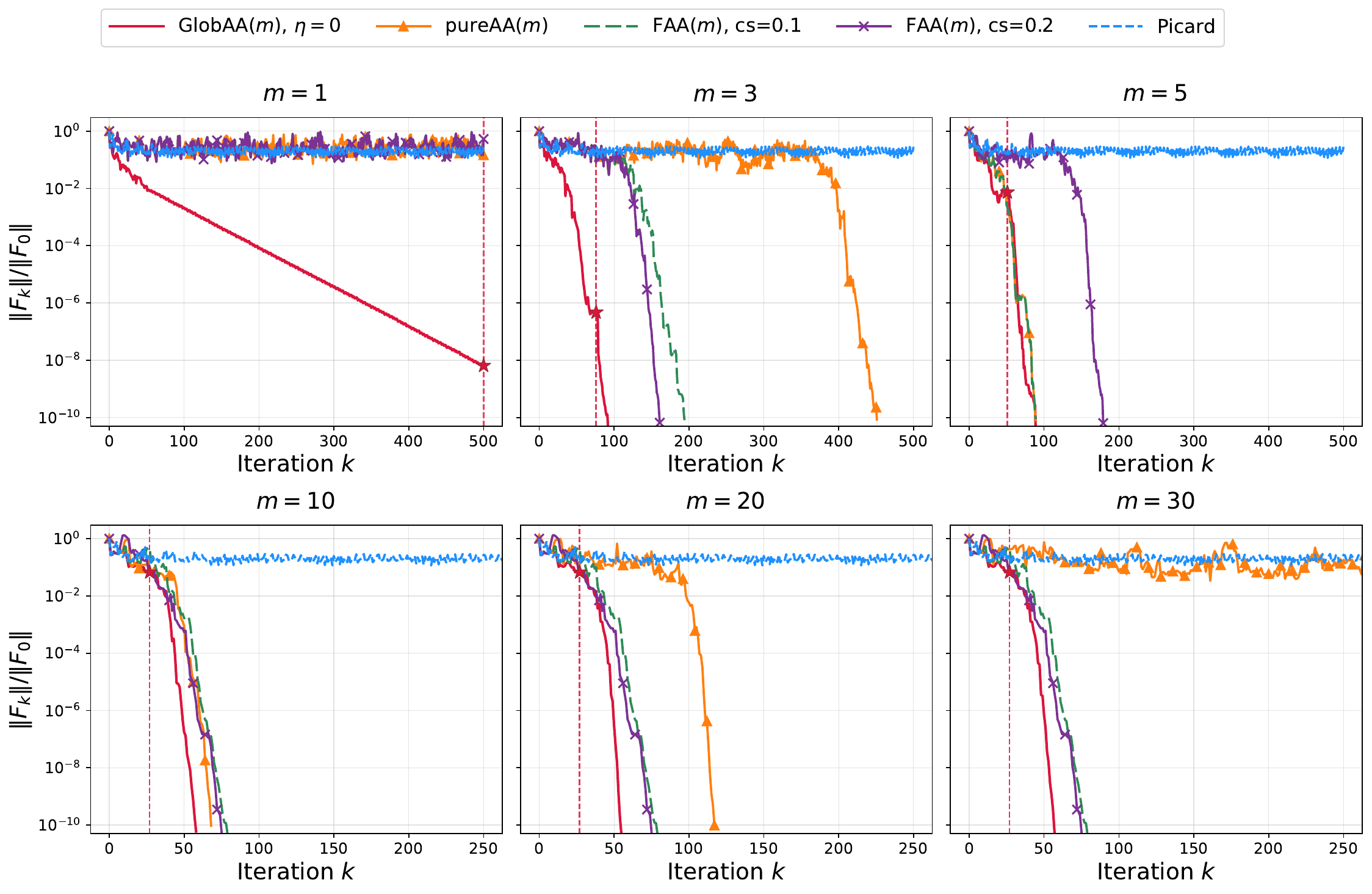}
    \caption{Convergence results of different algorithms with different memory sizes for NLH.}\label{fig3.3}
\end{figure}

\begin{table}[htbp]
\centering
\caption{Numerical results of different algorithms with different memory sizes for NLH }
\label{tab3.3:comparison}
\tiny
\setlength{\tabcolsep}{2pt}
\renewcommand{\arraystretch}{1.12}
\resizebox{\textwidth}{!}{
\begin{tabular}{lccc ccc ccc ccc ccc}
\toprule
\multirow{2}{*}{Algorithm}
& \multicolumn{3}{c}{$m=3$}
& \multicolumn{3}{c}{$m=5$}
& \multicolumn{3}{c}{$m=10$}
& \multicolumn{3}{c}{$m=20$}
& \multicolumn{3}{c}{$m=30$} \\
\cmidrule(lr){2-4}
\cmidrule(lr){5-7}
\cmidrule(lr){8-10}
\cmidrule(lr){11-13}
\cmidrule(lr){14-16}
& Iter. & Time(s) & AA(\%)
& Iter. & Time(s) & AA(\%)
& Iter. & Time(s) & AA(\%)
& Iter. & Time(s) & AA(\%)
& Iter. & Time(s) & AA(\%) \\
\midrule

Picard
& $500^{\dagger}$ & 0.29 & 0.00
& $500^{\dagger}$ & 0.29 & 0.00
& $500^{\dagger}$ & 0.29 & 0.00
& $500^{\dagger}$ & 0.29 & 0.00
& $500^{\dagger}$ & 0.29 & 0.00 \\

pureAA($m$)
& 451 & 0.48 & 100.00
& \textbf{89} & 0.11 & 100.00
& 68 & 0.16 & 100.00
& 117 & 0.54 & 100.00
& 341 & 2.62 & 100.00 \\

FAA($m$), $c_s=0.1$
& 194 & 0.26 & 100.00
& \textbf{89} & 0.13 & 100.00
& 79 & 0.13 & 100.00
& 79 & 0.13 & 100.00
& 79 & 0.14 & 100.00 \\

FAA($m$), $c_s=0.2$
& 161 & 0.21 & 100.00
& 179 & 0.26 & 100.00
& 75 & 0.15 & 100.00
& 75 & 0.15 & 100.00
& 75 & 0.16 & 100.00 \\

GlobAA($m$), $\eta=0$
& 92 & 0.09 & 78.26
& \textbf{89} & 0.11 & 85.39
& \textbf{58} & 0.10 & 75.86
& \textbf{55} & 0.13 & 74.55
& \textbf{57} & 0.15 & 75.44 \\

GlobAA($m$), $\eta=10^{-10}$
& \textbf{90} & \textbf{0.08} & 77.78
& 92 & \textbf{0.08} & 85.87
& \textbf{58} & \textbf{0.06} & 75.86
& \textbf{55} & \textbf{0.05} & 74.55
& \textbf{57} & \textbf{0.06} & 75.44 \\

\bottomrule
\end{tabular}
}
\end{table}

\subsection{2D Navier-Stokes equations}\label{sec3.4}
In this section, we consider the 2D lid-driven cavity flow governed by the steady Navier-Stokes equations (NSE) on the unit square $\Omega=(0,1)^2$, which can be expressed as follows:
\begin{equation*}
    \begin{aligned}
    \u \cdot \nabla \u + \nabla p - \mathrm{Re}^{-1}\Delta \u &= f, \\
    \nabla \cdot \u &= 0, \\
    \left.\u\right|_{\partial\Omega} &= \u_{{bc}}.
    \end{aligned}
\end{equation*}
The unknowns are the velocity field $\u$ and the pressure $p$, while $f$ denotes a prescribed external force. $\u_{bc}$ is a given (Dirichlet) boundary condition that is $(0,0)^\top$ on the sides and bottom and $(1,0)^\top$ on the lid. The parameter \rm{Re} denotes the Reynolds number, which is inversely proportional to the kinematic viscosity.

The Picard iteration for the NSE is given by (suppressing the spatial discretization)
\begin{equation}\label{2d_NS}
\begin{aligned}
\u_j \cdot \nabla {\u}_{j+1}
+ \nabla p_{j+1}
- \mathrm{Re}^{-1}\Delta {\u}_{j+1}
&= {f}, \\
\nabla \cdot {\u}_{j+1}
&= 0, \\
\left.{\u}_{j+1}\right|_{\partial\Omega}
&= {\u}_{{bc}}.
\end{aligned}
\end{equation}
The system above defines a fixed point iteration with $\u_{k+1} = G(\u_k)$, where $G$ is the solution operator for a spatial discretization of \eqref{2d_NS}. Following \cite{pollock2023filtering}, all algorithms are initialized with $\u_0 = G(0)$, which is the Stokes solution associated with the given problem data. In our experiments, there is no forcing ($f=0$).

\textbf{Implementation details.} In this experiment, we set $k_{\max} = 500$ for $\mathrm{Re}=5000$ and 8000, and $k_{\max} = 100$ for $\mathrm{Re}=10000$. Let ${\rm tol} = 10^{-10}$, $\lambda=0.5$ and $\gamma_k = 1 - \max\left\{\frac{1}{k+2}, \frac{W_k}{2 W_0}\right\}$ as in \eqref{eq_new_gamma} for GlobAA($m$). 
For FAA($m$), we use $\bar{\kappa}=10^8$ and test $c_s=0.2$ and $c_s=1/\sqrt{2}$.
% All other algorithms use the default parameters as mentioned in their paper. 
Following the discretization setting in \cite{pollock2023filtering}, we use grad-div stabilized (with parameter 1) Taylor--Hood $(P_2,P_1)$ finite elements for the spatial discretization. The computational mesh is constructed from a uniform $N\times N$ triangulation of the unit square and is further refined once around the edges (within $0.1$ from the boundary). Different values of $N$ yield different degrees of freedom (DOFs). Similar to subsection~\ref{sec3.3}, we also run GlobAA($m$) using \eqref{eq3.2} with regularization parameter $\eta = 10^{-10}$. 
The cost of computing the common initial iterate $G(0)$ is excluded from the reported CPU times.

\textbf{Experimental results.}
Fig.~\ref{fig3.4-1} and Table~\ref{tab3.5:cavity} show that the Picard iteration stagnates in all tested cases, whereas the AA-based methods successfully reduce the residual to the prescribed tolerance. For clarity, only the first 100 iterations are shown in Fig.~\ref{fig3.4-1}.
GlobAA($m$) requires the fewest or nearly the fewest iterations across different memory sizes and remains robust across the tested combinations of Reynolds numbers and problem sizes.
Moreover, the regularized variant with $\eta=10^{-10}$ exhibits almost identical convergence behavior while generally reducing the computational time. The residual curves further confirm the fast convergence of GlobAA($m$), while Fig.~\ref{fig3.4-2} displays the expected dominant cavity vortex and corner recirculation structures at different Reynolds numbers.

\begin{figure}[htbp]
    \centering
    \includegraphics[width=0.99\textwidth]{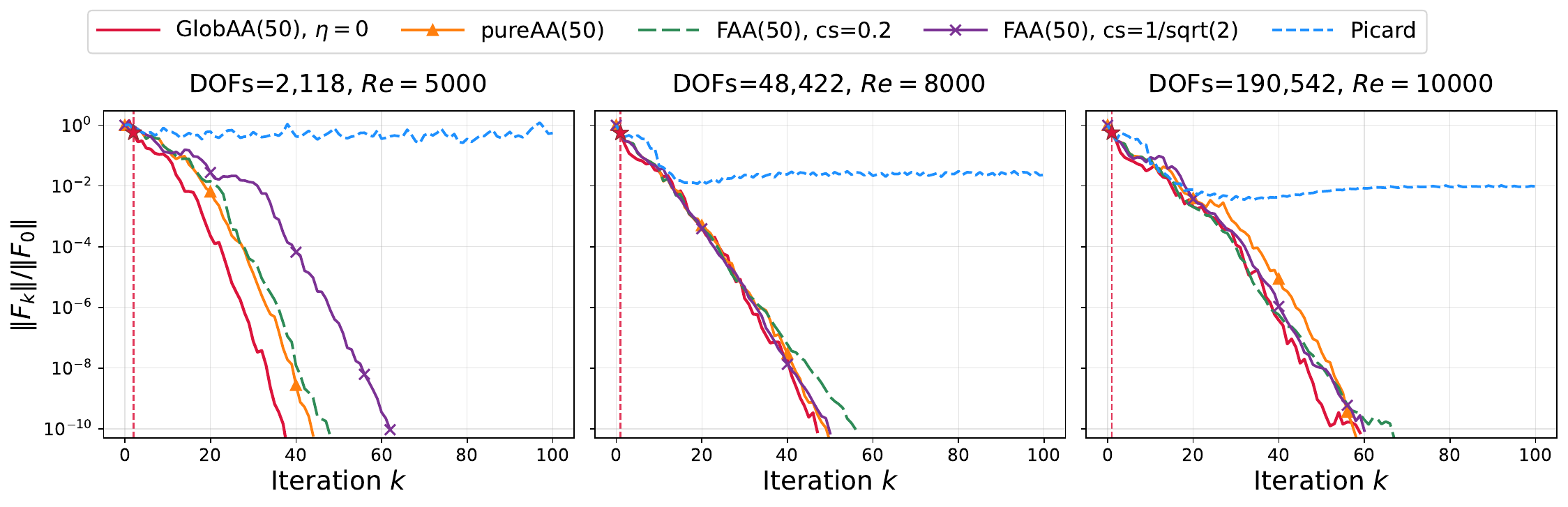}
    \caption{Convergence histories of different algorithms with $m=50$ for NSE under different Reynolds numbers and discretization sizes.}\label{fig3.4-1}
\end{figure}

\begin{figure}[htbp]
    \centering
    \includegraphics[width=0.99\textwidth]{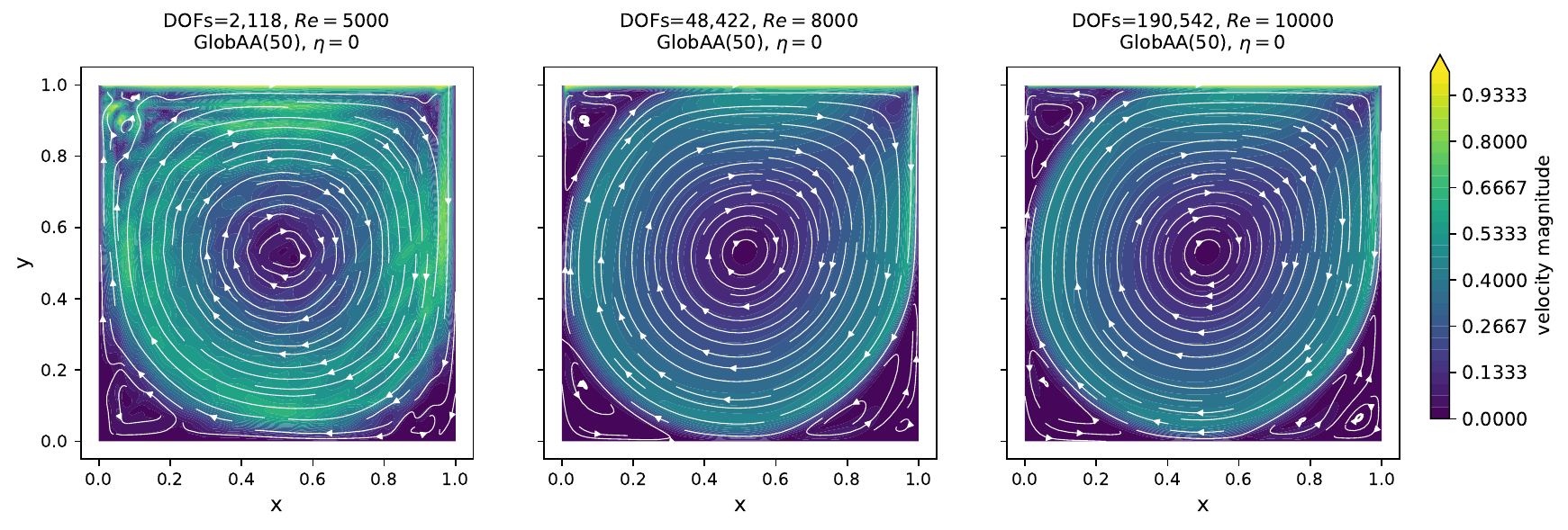}
    \caption{Solutions of the steady 2D-NSE driven cavity problem for different parameters.}\label{fig3.4-2}
\end{figure}

% 统一框架
\begin{table}[htbp]
\centering
\caption{Numerical results of different algorithms for NSE with different parameters}
\label{tab3.5:cavity}
\tiny
\setlength{\tabcolsep}{2pt}
\renewcommand{\arraystretch}{1.12}

\resizebox{\textwidth}{!}{
\begin{tabular}{lllccc ccc ccc ccc}
\toprule
\multirow{2}{*}{Parameters}
& \multirow{2}{*}{DOFs}
& \multirow{2}{*}{Algorithm}
& \multicolumn{3}{c}{$m=10$}
& \multicolumn{3}{c}{$m=20$}
& \multicolumn{3}{c}{$m=30$}
& \multicolumn{3}{c}{$m=50$} \\
\cmidrule(lr){4-6}
\cmidrule(lr){7-9}
\cmidrule(lr){10-12}
\cmidrule(lr){13-15}
&
&
& Iter. & Time(s) & AA(\%)
& Iter. & Time(s) & AA(\%)
& Iter. & Time(s) & AA(\%)
& Iter. & Time(s) & AA(\%) \\
\midrule

% Re = 5000, N = 10
\multirow{6}{*}{\shortstack{$\mathrm{Re}=5000$\\ $N=10$}}
& \multirow{6}{*}{2,118}
& Picard
& $500^{\dagger}$ & 7.55 & 0.0
& $500^{\dagger}$ & 7.55 & 0.0
& $500^{\dagger}$ & 7.55 & 0.0
& $500^{\dagger}$ & 7.55 & 0.0 \\

&
& pureAA($m$)
& 46 & 0.77 & 100.0
& 43 & 0.76 & 100.0
& 42 & 0.81 & 100.0
& 44 & 0.88 & 100.0 \\

&
& FAA($m$), $c_s=0.2$
& 48 & 0.77 & 100.0
& 48 & 0.78 & 100.0
& 48 & 0.82 & 100.0
& 48 & 0.81 & 100.0 \\

&
& FAA($m$), $c_s=1/\sqrt{2}$
& 62 & 1.00 & 100.0
& 62 & 1.02 & 100.0
& 62 & 1.04 & 100.0
& 62 & 1.03 & 100.0 \\

&
& GlobAA($m$), $\eta=0$
& \textbf{41} & 0.72 & 95.1
& \textbf{37} & 0.68 & 94.6
& \textbf{37} & 0.72 & 94.6
& \textbf{38} & 0.72 & 94.7 \\

&
& GlobAA($m$), $\eta=10^{-10}$
& \textbf{41} & \textbf{0.64} & 95.1
& \textbf{37} & \textbf{0.63} & 94.6
& \textbf{37} & \textbf{0.65} & 94.6
& \textbf{38} & \textbf{0.67} & 94.7 \\

\midrule

% Re = 8000, N = 50
\multirow{6}{*}{\shortstack{$\mathrm{Re}=8000$\\ $N=50$}}
& \multirow{6}{*}{48,422}
& Picard
& $500^{\dagger}$ & 900.23 & 0.0
& $500^{\dagger}$ & 900.23 & 0.0
& $500^{\dagger}$ & 900.23 & 0.0
& $500^{\dagger}$ & 900.23 & 0.0 \\

&
& pureAA($m$)
& 54 & 99.51 & 100.0
& \textbf{49} & \textbf{89.78} & 100.0
& 48 & 90.33 & 100.0
& 50 & 92.24 & 100.0 \\

&
& FAA($m$), $c_s=0.2$
& 56 & 100.35 & 100.0
& 56 & 100.02 & 100.0
& 56 & 101.43 & 100.0
& 56 & 102.12 & 100.0 \\

&
& FAA($m$), $c_s=1/\sqrt{2}$
& 55 & 98.18 & 100.0
& 50 & 90.28 & 100.0
& 50 & 89.82 & 100.0
& 50 & 90.22 & 100.0 \\

&
& GlobAA($m$), $\eta=0$
& \textbf{52} & 95.57 & 98.1
& 50 & 93.57 & 98.0
& \textbf{47} & 89.63 & 97.9
& \textbf{47} & 89.83 & 97.9 \\

&
& GlobAA($m$), $\eta=10^{-10}$
& \textbf{52} & \textbf{94.72} & 98.1
& 50 & 90.80 & 98.0
& \textbf{47} & \textbf{86.02} & 97.9
& 48 & \textbf{87.97} & 97.9 \\

\midrule

% Re = 10000, N = 100
\multirow{6}{*}{\shortstack{$\mathrm{Re}=10000$\\ $N=100$}}
& \multirow{6}{*}{190,542}
& Picard
& $100^{\dagger}$ & 1596.15 & 0.0
& $100^{\dagger}$ & 1596.15 & 0.0
& $100^{\dagger}$ & 1596.15 & 0.0
& $100^{\dagger}$ & 1596.15 & 0.0 \\

&
& pureAA($m$)
& \textbf{62} & \textbf{1010.64} & 100.0
& 59 & 939.83 & 100.0
& 57 & 916.28 & 100.0
& \textbf{58} & \textbf{931.18} & 100.0 \\

&
& FAA($m$), $c_s=0.2$
& 67 & 1073.25 & 100.0
& 67 & 1061.83 & 100.0
& 67 & 1060.42 & 100.0
& 67 & 1069.25 & 100.0 \\

&
& FAA($m$), $c_s=1/\sqrt{2}$
& 64 & 1027.11 & 100.0
& 60 & 958.57 & 100.0
& 60 & 955.42 & 100.0
& 60 & 949.74 & 100.0 \\

&
& GlobAA($m$), $\eta=0$
& 63 & 1024.09 & 98.4
& \textbf{54} & 891.68 & 98.1
& \textbf{52} & 857.28 & 98.1
& 59 & 963.06 & 98.3 \\

&
& GlobAA($m$), $\eta=10^{-10}$
& 63 & 1034.71 & 98.4
& \textbf{54} & \textbf{885.26} & 98.1
& \textbf{52} & \textbf{842.66} & 98.1
& 59 & 970.42 & 98.3 \\

\bottomrule
\end{tabular}
}
\end{table}

\section{Conclusion}
We propose a nonmonotone globalized Anderson acceleration framework with a novel globalization strategy, termed GlobAA($m$). Under mild assumptions, we establish convergence results that combine global convergence guarantees with local acceleration. More specifically, for contractive mappings, we prove global residual and iterate $r$-linear convergence for $m \ge 1$, and global residual $q$-linear convergence for $m=1$, with convergence factors no greater than the corresponding contraction factor. Moreover, we show that GlobAA($m$) reduces to pure Anderson acceleration after finitely many iterations. For nonexpansive mappings, we establish the global convergence of the fixed point residuals to zero. Finally, we construct a class-agnostic parameter selection rule that ensures global residual convergence in the nonexpansive setting while preserving linear convergence in the contractive setting.
Compared with existing globalization strategies, GlobAA($m$) requires little problem-specific prior knowledge and limited parameter tuning, making it broadly applicable. Extensive numerical experiments on various problems demonstrate its effectiveness and robustness, particularly in challenging scenarios where pure Anderson acceleration may fail to converge or stagnate.

% ---------- 参考文献 ----------
% 将自己的 .bib 文件上传到同一项目，取消下面第二行的注释，
% 并把 references 改成实际的 .bib 文件名（不含扩展名）。
\bibliographystyle{plain}
\bibliography{references}

@article{chen2019convergence,
  title={Convergence of the {EDIIS} algorithm for nonlinear equations},
  author={Chen, Xiaojun and Kelley, Carl T},
  journal={SIAM J. Sci. Comput.},
  volume={41},
  number={1},
  pages={A365--A379},
  year={2019},
  publisher={SIAM}
}

@article{bian2022anderson,
  title={Anderson acceleration for nonsmooth fixed point problems},
  author={Bian, Wei and Chen, Xiaojun},
  journal={SIAM J. Numer. Anal.},
  volume={60},
  number={5},
  pages={2565--2591},
  year={2022},
  publisher={SIAM}
}

@article{Anderson,
  title={Iterative procedures for nonlinear integral equations},
  author={Anderson, Donald G},
  journal={J. ACM},
  volume={12},
  number={4},
  pages={547--560},
  year={1965},
  publisher={ACM New York, NY, USA}
}

@article{walker2011anderson,
  title={Anderson acceleration for fixed-point iterations},
  author={Walker, Homer F and Ni, Peng},
  journal={SIAM J. Numer. Anal.},
  volume={49},
  number={4},
  pages={1715--1735},
  year={2011},
  publisher={SIAM}
}

@article{bian2021anderson,
  title={Anderson acceleration for a class of nonsmooth fixed-point problems},
  author={Bian, Wei and Chen, Xiaojun and Kelley, Carl T},
  journal={SIAM J. Sci. Comput.},
  volume={43},
  number={5},
  pages={S1--S20},
  year={2021},
  publisher={SIAM}
}

@article{TothKelley2015,
  title={Convergence analysis for {Anderson} acceleration},
  author={Toth, Alex and Kelley, Carl T},
  journal={SIAM J. Numer. Anal.},
  volume={53},
  number={2},
  pages={805--819},
  year={2015},
  publisher={SIAM}
}

@article{zhang2020globally,
  title={Globally convergent type-{I} {Anderson} acceleration for nonsmooth fixed-point iterations},
  author={Zhang, Junzi and O'Donoghue, Brendan and Boyd, Stephen},
  journal={SIAM J. Optim.},
  volume={30},
  number={4},
  pages={3170--3197},
  year={2020},
  publisher={SIAM}
}

@article{evans2020proof,
  title={A proof that {Anderson} acceleration improves the convergence rate in linearly converging fixed-point methods (but not in those converging quadratically)},
  author={Evans, Claire and Pollock, Sara and Rebholz, Leo G and Xiao, Mengying},
  journal={SIAM J. Numer. Anal.},
  volume={58},
  number={1},
  pages={788--810},
  year={2020},
  publisher={SIAM}
}

@article{pollock2021anderson,
  title={Anderson acceleration for contractive and noncontractive operators},
  author={Pollock, Sara and Rebholz, Leo G},
  journal={IMA J. Numer. Anal.},
  volume={41},
  number={4},
  pages={2841--2872},
  year={2021},
  publisher={Oxford University Press}
}

@article{ouyang2024descent,
  title={Descent Properties of an {Anderson} Accelerated Gradient Method with Restarting},
  author={Ouyang, Wenqing and Liu, Yang and Milzarek, Andre},
  journal={SIAM J. Optim.},
  volume={34},
  number={1},
  pages={336--365},
  year={2024},
  publisher={SIAM}
}

@article{ouyang2023nonmonotone,
  title={Nonmonotone globalization for {Anderson} acceleration via adaptive regularization},
  author={Ouyang, Wenqing and Tao, Jiong and Milzarek, Andre and Deng, Bailin},
  journal={J. Sci. Comput.},
  volume={96},
  number={1},
  pages={5},
  year={2023},
  publisher={Springer}
}

@article{fu2020anderson,
  title={Anderson Accelerated {Douglas--Rachford} Splitting},
  author={Fu, Anqi and Zhang, Junzi and Boyd, Stephen},
  journal={SIAM J. Sci. Comput.},
  volume={42},
  number={6},
  pages={A3560--A3583},
  year={2020},
  publisher={SIAM}
}

@inproceedings{mai2020anderson,
  title={Anderson acceleration of proximal gradient methods},
  author={Mai, Vien and Johansson, Mikael},
  booktitle={Proceedings of the 37th International Conference on Machine Learning, H. Daumé III and A. Singh, eds., vol. 119 of Proceedings of Machine Learning Research},
  pages={6620--6629},
  year={13-18 Jul 2020},
  organization={PMLR}
}

@article{de2022linear,
  title={Linear asymptotic convergence of Anderson acceleration: fixed-point analysis},
  author={De Sterck, Hans and He, Yunhui},
  journal={SIAM J. Matrix Anal. Appl.},
  volume={43},
  number={4},
  pages={1755--1783},
  year={2022},
  publisher={SIAM}
}

@article{krzysik2025asymptotic,
  title={Asymptotic convergence of restarted {Anderson} acceleration for certain normal linear systems},
  author={Krzysik, Oliver A and De Sterck, Hans and Smith, Adam},
  journal={SIAM J. Sci. Comput.},
  pages={S135--S160},
  year={2025},
  publisher={SIAM}
}

@article{pollock2019anderson,
  title={Anderson-accelerated convergence of {Picard} iterations for incompressible {Navier--Stokes} equations},
  author={Pollock, Sara and Rebholz, Leo G and Xiao, Mengying},
  journal={SIAM J. Numer. Anal.},
  volume={57},
  number={2},
  pages={615--637},
  year={2019},
  publisher={SIAM}
}

@article{banerjee2016periodic,
  title={Periodic {Pulay} method for robust and efficient convergence acceleration of self-consistent field iterations},
  author={Banerjee, Amartya S and Suryanarayana, Phanish and Pask, John E},
  journal={Chem. Phys. Lett.},
  volume={647},
  pages={31--35},
  year={2016},
  publisher={Elsevier}
}

@article{fang2009two,
  title={Two classes of multisecant methods for nonlinear acceleration},
  author={Fang, Haw-Ren and Saad, Yousef},
  journal={Numer. Linear Algebra Appl.},
  volume={16},
  number={3},
  pages={197--221},
  year={2009},
  publisher={Wiley Online Library}
}

@article{pollock2023filtering,
  title={Filtering for {Anderson} acceleration},
  author={Pollock, Sara and Rebholz, Leo G},
  journal={SIAM J. Sci. Comput.},
  volume={45},
  number={4},
  pages={A1571--A1590},
  year={2023},
  publisher={SIAM}
}

@article{baruch2007high,
  title={High-order numerical method for the nonlinear {Helmholtz} equation with material discontinuities in one space dimension},
  author={Baruch, Guy and Fibich, Gadi and Tsynkov, Semyon},
  journal={J. Comput. Phys.},
  volume={227},
  number={1},
  pages={820--850},
  year={2007},
  publisher={Elsevier}
}

@article{baruch2009high,
  title={A high-order numerical method for the nonlinear {Helmholtz} equation in multidimensional layered media},
  author={Baruch, Guy and Fibich, Gadi and Tsynkov, Semyon},
  journal={J. Comput. Phys.},
  volume={228},
  number={10},
  pages={3789--3815},
  year={2009},
  publisher={Elsevier}
}

@article{o2016conic,
  title={Conic optimization via operator splitting and homogeneous self-dual embedding},
  author={O'Donoghue, Brendan and Chu, Eric and Parikh, Neal and Boyd, Stephen},
  journal={J. Optim. Theory Appl.},
  volume={169},
  number={3},
  pages={1042--1068},
  year={2016},
  publisher={Springer}
}

@article{anderson2019comments,
  title={Comments on “{Anderson} acceleration, mixing and extrapolation”},
  author={Anderson, Donald GM},
  journal={Numer. Algorithms},
  volume={80},
  number={1},
  pages={135--234},
  year={2019},
  publisher={Springer}
}

@article{chupin2021convergence,
  title={Convergence analysis of adaptive {DIIS} algorithms with application to electronic ground state calculations},
  author={Chupin, Maxime and Dupuy, Mi-Song and Legendre, Guillaume and S{\'e}r{\'e}, Eric},
  journal={ESAIM Math. Model. Numer. Anal.},
  volume={55},
  number={6},
  pages={2785--2825},
  year={2021},
  publisher={EDP Sciences}
}

@article{pulay1980convergence,
  title={Convergence acceleration of iterative sequences. {The} case of {SCF} iteration},
  author={Pulay, P{\'e}ter},
  journal={Chem. Phys. Lett.},
  volume={73},
  number={2},
  pages={393--398},
  year={1980},
  publisher={Elsevier}
}

@article{sterck2021asymptotic,
  title={On the asymptotic linear convergence speed of {Anderson} acceleration, {Nesterov} acceleration, and nonlinear {GMRES}},
  author={De Sterck, Hans and He, Yunhui},
  journal={SIAM J. Sci. Comput.},
  volume={43},
  number={5},
  pages={S21--S46},
  year={2021},
  publisher={SIAM}
}

@article{sterck2012nonlinear,
  title={A nonlinear {GMRES} optimization algorithm for canonical tensor decomposition},
  author={De Sterck, Hans},
  journal={SIAM J. Sci. Comput.},
  volume={34},
  number={3},
  pages={A1351--A1379},
  year={2012},
  publisher={SIAM}
}

@article{peng2018anderson,
  title={Anderson acceleration for geometry optimization and physics simulation},
  author={Peng, Yue and Deng, Bailin and Zhang, Juyong and Geng, Fanyu and Qin, Wenjie and Liu, Ligang},
  journal={ACM Trans. Graph.},
  volume={37},
  number={4},
  year={2018},
  pages={42:1--42:14},
  publisher={ACM New York, NY, USA}
}

@article{lipnikov2013anderson,
  title={Anderson acceleration for nonlinear finite volume scheme for advection-diffusion problems},
  author={Lipnikov, Konstantin and Svyatskiy, Daniil and Vassilevski, Y},
  journal={SIAM J. Sci. Comput.},
  volume={35},
  number={2},
  pages={A1120--A1136},
  year={2013},
  publisher={SIAM}
}

@article{pratapa2016anderson,
  title={Anderson acceleration of the {Jacobi} iterative method: An efficient alternative to {Krylov} methods for large, sparse linear systems},
  author={Pratapa, Phanisri P and Suryanarayana, Phanish and Pask, John E},
  journal={J. Comput. Phys.},
  volume={306},
  pages={43--54},
  year={2016},
  publisher={Elsevier}
}

@article{an2017anderson,
  title={Anderson acceleration and application to the three-temperature energy equations},
  author={An, Hengbin and Jia, Xiaowei and Walker, Homer F},
  journal={J. Comput. Phys.},
  volume={347},
  pages={1--19},
  year={2017},
  publisher={Elsevier}
}

@article{both2019anderson,
  title={Anderson accelerated fixed-stress splitting schemes for consolidation of unsaturated porous media},
  author={Both, Jakub Wiktor and Kumar, Kundan and Nordbotten, Jan Martin and Radu, Florin Adrian},
  journal={Comput. Math. Appl.},
  volume={77},
  number={6},
  pages={1479--1502},
  year={2019},
  publisher={Elsevier}
}

@article{matveev2018anderson,
  title={Anderson acceleration method of finding steady-state particle size distribution for a wide class of aggregation--fragmentation models},
  author={Matveev, Sergey A and Stadnichuk, V I and Tyrtyshnikov, E E and Smirnov, Alexander P and Ampilogova, N V and Brilliantov, Nikolai V},
  journal={Comput. Phys. Commun.},
  volume={224},
  pages={154--163},
  year={2018},
  publisher={Elsevier}
}

@inproceedings{wei2022class,
  title={A class of short-term recurrence {Anderson} mixing methods and their applications},
  author={Wei, Fuchao and Bao, Chenglong and Liu, Yang},
  booktitle={International Conference on Learning Representations},
  year={2022}
}

@article{scieur2020regularized,
  title={Regularized nonlinear acceleration},
  author={Scieur, Damien and d’Aspremont, Alexandre and Bach, Francis},
  journal={Math. Program.},
  volume={179},
  number={1},
  pages={47--83},
  year={2020},
  publisher={Springer}
}

@article{zou2005regularization,
  title={Regularization and variable selection via the elastic net},
  author={Zou, Hui and Hastie, Trevor},
  journal={J. R. Stat. Soc. Ser. B Stat. Methodol.},
  volume={67},
  number={2},
  pages={301--320},
  year={2005},
  publisher={Oxford University Press}
}

@article{beck2009fast,
  title={A fast iterative shrinkage-thresholding algorithm for linear inverse problems},
  author={Beck, Amir and Teboulle, Marc},
  journal={SIAM J. Imaging Sci.},
  volume={2},
  number={1},
  pages={183--202},
  year={2009},
  publisher={SIAM}
}

@article{paige1982lsqr,
  title={{LSQR}: An algorithm for sparse linear equations and sparse least squares},
  author={Paige, Christopher C and Saunders, Michael A},
  journal={ACM Trans. Math. Softw.},
  volume={8},
  number={1},
  pages={43--71},
  year={1982},
  publisher={ACM New York, NY, USA}
}

% 如果原稿使用手动 thebibliography 环境，可直接将其粘贴到这里，
% 并删除上面的 bibliographystyle / bibliography 两行。

\end{document}